\documentclass[final]{siamltex}

\usepackage{amsfonts}

\newtheorem{assump}{Assumption}

\title{Impulse Control of Multi-dimensional Jump Diffusions in Finite Time Horizon}

\author{Yann-Shin Aaron Chen \thanks{Department of Mathematics, University of California at Berkeley, CA 94720-3840. Email address: \texttt{yac@math.berkeley.edu}}
\and Xin Guo \thanks{Department of Industrial Engineering and Operations Research, University of California at Berkeley, CA 94720-1777. Email address: \texttt{xinguo@ieor.berkeley.edu}}  }

\begin{document}

\maketitle

\begin{abstract}\noindent
This paper analyzes a class of impulse control problems for multi-dimensional jump diffusions in the finite time horizon. Following the basic mathematical setup from Stroock and Varadhan \cite{StroockVaradhan06}, this paper first establishes rigorously an appropriate form of Dynamic Programming Principle (DPP). It then shows that the value function is a  viscosity solution for the associated Hamilton-Jacobi-Belleman (HJB) equation involving integro-differential operators. Finally, under additional assumptions that the jumps are of infinite activity but are of finite variation and that the diffusion is uniformly elliptic, it proves that the value function is the unique viscosity solution and has $W_{loc}^{(2,1),p}$ regularity for $1< p< \infty$.
\end{abstract}



\section{Introduction}
This paper considers a class of impulse control problem for an
$n$-dimensional diffusion process $X_t$ in the form of Equation (\ref{LevySDE}).
The objective is to choose an appropriate  impulse control
so that a certain class of cost function in the form of (\ref{problem}) can be minimized.

Impulse control,
in contrast to regular and singular controls, allows the state space to be discontinuous and is  a more
natural mathematical framework  for many applied problems in engineering and economics.
Examples in financial mathematics include  portfolio management with transaction costs \cite{BP00, Korn98, Korn99, EH88,
MP95, OS02}, insurance models \cite{JS95, CCTZ06}, liquidity
risk \cite{VMP07},  optimal control of exchange rates
\cite{Jeanblanc93, MO98, CZ99}, and  real options
\cite{TH90, MT94}.
 Similar to their regular and singular control counterparts, impulse control problems can be analyzed via various approaches.  One approach is to focus on solving the value function for the associated (quasi)-variational inequalities or Hamilton-Jacobi-Bellman (HJB) integro-differential equations, and then establishing the optimality of the solution by verification theorem. (See $\O$ksendal and Sulem \cite{OS04}.)
 Another approach is to characterize the value function of the control problem as a (unique) viscosity solution to the associated PDEs, and/or to study their regularities. In both approaches, in order to connect the PDEs and the original control problems, some form of the Dynamic Programming Principle (DPP) is usually implicitly or explicitly assumed.

Compared to the regular and the singular controls, the main difficulty with impulse controls is the associated non-local operator, which is more difficult to analyze via the classical PDEs tools. When jumps are added to the diffusion, one also has to deal with an integral-differential operator instead of the differential operator. The earliest mathematical literature on impulse controls is the well-known book by Bensoussan and Lions  \cite{BL82}, where value functions of the control problems for diffusions without jumps were shown to satisfy the
quasi-variational-inequalities and where their regularity properties were
established when the control is strictly positive and the state
space is in a bounded region. See also work by Menaldi \cite{Menaldi80}, \cite{Menaldi81} and \cite{Menaldi87} for jump diffusions with degeneracy. Recently,
Barles, Chasseigne, and Imbert \cite{BCI08} provided a general framework and careful analysis on the unique viscosity solution  of second order nonlinear elliptic/parabolic partial integro-differential equations. However, in all these PDEs-focused  papers, the DPP, the essential link between the PDEs and  control problems, is missing.
 On the other hand, Tang and Yong \cite{TY93} and Ishikawa \cite{Ishikawa04} established some version of the DPP and the uniqueness of the viscosity solution  for  diffusions  without jumps.  More recently, Seydel \cite{Seydel09}  used a version of the DPP for Markov controls to study the viscosity solution of control problems on jump diffusions. This Markovian assumption simplifies the proof of DPP significantly.
 Based on the Markovian setup of \cite{Seydel09}, Davis, Guo and Wu \cite{DGW10} focused on the regularities of the viscosity solution associated with the control problems on jump diffusions in an infinite time horizon. \cite{DGW10} extended some techniques developed in  Guo and Wu \cite{GW09} and used the key  connection between the non-local operator and the differential operator developed in \cite{GW09}.

In essence, there are three aspects when studying impulse control problems: the DPP, the HJB, and the regularity of the value function. However, all  previous work addressed only one or two of the above aspects and used quite different setups and assumptions, making it difficult to see exactly to what extent all the relevant properties hold in a given framework. This is the motivation of our paper.

\paragraph{Our Results.}

This paper studies the finite time horizon impulse control problem of Eq. (\ref{problem}) on multi-dimensional jump diffusions of Eq. \ref{LevySDE}. Within the same mathematical framework, this paper study  three aspects of the control problem: the DPP, the viscosity solution, and its uniqueness and regularity.

\begin{romannum}
\item First, it takes  the classical setup of Stroock and Varadhan \cite{StroockVaradhan06}, assumes the natural filtration of the underlying Brownian motion and the Poisson process, and establishes a general form of DPP.  This natural filtration is different from the  ``usual hypothesis'', i.e., the  completed right continuous filtration assumed {\it a priori} in most existing works. This specification ensures the existence and certain properties for the existence of the regular conditional probability, crucial for rigorously establishing towards DPP. (See Lemma 4.3.3 from Stroock \& Varadhan \cite{StroockVaradhan06}).
    With additional appropriate  estimation for the value function, the DPP is proved.

We remark that  there are some previous works of DPP for impulse controls.  For instance,  \cite{TY93} proved a form of DPP for diffusions without jumps and \cite{Seydel09}  restricted controls to be only Markov controls. Because of the inclusion of both jumps and  non-Markov controls, there are essential mathematical difficulty for establishing the DPP and hence the necessity to adopt the framework of \cite{StroockVaradhan06}.

Note that an alternative approach would be  to adopt the general weak DPP formulation by Bouchard and Touzi \cite{BT11} and Bouchard and Nutz \cite{BN11}, or the classical work by El Karoui \cite{ElKouri81}. However, verifying the key assumptions of the weak DPP, especially the ``flow property" in a controlled jump diffusion framework, does not appear simpler than directly establishing the regular conditional probability.

\item Second, it shows that  the value function is a  viscosity solution in the sense of \cite{BCI08}. This form of viscosity solution is convenient for the HJB equations involving integro-differential operators, which is the key for analyzing  control problems on jump diffusions.

Again, special cases have been studied in  \cite{Seydel09}  for the Markov controls  and \cite{TY93} for diffusions without jumps.

\item  Third, under additional assumption that the jumps are of finite variation with possible infinite activity, it proves the $W^{(2,1),p}_{loc}$ regularity and  the unique viscosity solution properties of the value function.  Note that the uniqueness of the viscosity solution in our paper is a ``local'' uniqueness, which is appropriate to study the  regularity property.

 Compared to \cite{GW09} without jumps and especially \cite{DGW10} with jumps for infinite horizon problems, this paper is on a finite time horizon which requires  different techniques. First, it is more difficult in a parabolic setting to obtain a priori  estimates for the value function of the stochastic control problem, especially with the relaxed assumption of the H\"older growth condition (Outstanding Assumption 4 in our paper). Our estimation extends the earlier work of \cite{TY93} to diffusions with jumps.
Secondly, from a PDE perspective, we  introduce the notion of H\"older continuity on the measure (Assumption 11). We believe that Assumptions 10 and 11 are more general than those in \cite{DGW10}, and are consistent with the  approach in \cite{StroockVaradhan06} with the focus on the integro-differential operator itself. Finally,  \cite{GW09} \cite{DGW10}  studies neither the DPP nor the uniqueness of the viscosity solution.
  There were  also studies by Xing and Bayraktar \cite{BX09} and Pham \cite{Pham98} on value functions for optimal stopping problems for jump diffusions. Their work however did not involve controls. \footnote{We are  brought to attention by one of the referees some very recent and nice work by \cite{BX12} and \cite{BEM12} regarding the regularity analysis for optimal stopping and impulse control problems with infinite variation jumps.}


\end{romannum}


\section{Problem Formulation and Main Results}


\subsection{Problem formulation}

\paragraph{Filtration} Fix a time $T> 0$. For each $t_0 \in [0, T]$, let $(\Omega, {\cal F}, \mathbb{P})$ be a probability space  that supports a Brownian motion $\{ W_\cdot \} _{t_0 \leq s \leq T}$ starting at $t_0$, and an independent Poisson point process $N(dt, dz)$ on $([t_0, T], \mathbb{R}^k \setminus \{0\})$ with intensity $L \otimes \rho$. Here $L$ is the Lebesgue measure on $[t_0, T]$ and $\rho$ is a measure defined on $\mathbb{R}^k \setminus \{0\}$.
For each $t \in [t_0, T]$, define $\{{\cal F}_t^{W,N}\}_{t_0 \leq t \leq T}$ to be the natural filtration of the Brownian motion $W$ and the Poisson process $N$,  define  $\{{\cal F}_{t_0, t}[t_0, T]\}$ to be $\{{\cal F}_t^{W,N}\}_{t_0 \leq t \leq T}$ restricted to the interval $[t_0, t]$.

Throughout the paper, we will use  this uncompleted natural filtration $\{{\cal F}_{t_0, t}[t_0, T]\}$.

Now, we can define mathematically the  impulse control problem,  starting with the set of admissible controls.
\begin{definition}
The set of admissible impulse control ${\cal V} [t_0, T]$ consists of pairs of sequences $\{ \tau_i, \xi_i\}_{1 \leq i < \infty}$ such that
\begin{enumerate}
\item
$\tau_i: \Omega \to [t_0, T] \cup \{\infty\}$ such that $\tau_i$ are stopping times with respect to the filtration $\{{\cal F}^{W, N}_{t_0, s}\}_{t_0 \leq s \leq T}$,
\item $\tau_i \leq \tau_{i+1}$ for all $i$,
\item $\xi_i: \Omega \to \mathbb{R}^n \setminus \{ 0\}$ is a random variable such that $\xi_i \in {\cal F}^{W,N}_{t_0, \tau_i}$.
\end{enumerate}
\end{definition}

Now, given an admissible impulse control $\{ \tau_i, \xi_i\}_{1 \leq i < \infty}$,  a stochastic process $(X_t)_{t\ge 0}$ follows a  stochastic differential equation with jumps,
\begin{eqnarray}\label{LevySDE}
X_t   &=&x_0 +  \int_{t_0}^t b(X_{s-}, s) ds + \int_{t_0}^t \sigma (X_{s-}, s) dW_s  + \sum_i \xi_i 1_{( \tau_i \leq t)} \nonumber\\
&&+ \int_{t_0}^t \int j_1(X_{s-}, s, z) N(dz, dt) + \int_{t_0}^t \int j_2(X_{s-}, s, z) \widetilde{N}(dz, dt).
\end{eqnarray}
 Here $\widetilde{N} = N (dt, dz) - \rho(dz) dt$,  $b: \mathbb{R}^n \times [0, T] \times \mathbb{R}^n$, $\sigma : \mathbb{R}^n \times [0, T] \to \mathbb{R}^{n \times m}$, and $j_1, j_2: \mathbb{R}^n \times [0, T]  \times \mathbb{R}^k \to \mathbb{R}^n$. For each $(\tau_i, \xi_i)_i \in {\cal V}[t_0, T]$, and $(x_0, t_0) \in \mathbb{R}^n \times [t_0, T]$, denote $X = X^{t_0, x_0, u, \tau_i, \xi_i}$.

The stochastic control problem is to
\begin{equation}\label{problem}
\mbox{(Problem)} \ \ \mbox{Minimize} \ \ \ J[x_0, t_0, \tau_i, \xi_i] \ \ \mbox{ over all $(\tau_i, \xi_i) \in {\cal V}[t, T]$},
\end{equation}
subject to Eqn. (\ref{LevySDE}) with
\begin{equation}
 J[x_0, t_0, \tau_i, \xi_i] =\mathbb{E} \left[ \int_{t_0}^T f(s, X_s^{t_0, x_0, \tau_i, \xi_i}) ds +g( X_T^{t_0, x_0,\tau_i, \xi_i}) \right]
+  \mathbb{E} \left[\sum_i B(\xi_i, \tau_i) 1_{\{t_0 \leq \tau_i \leq T\}}\right].
\end{equation}
Here we denote $V$ for the associated value function
\begin{equation}\label{eqnvalue}
\mbox{(Value Function) } \ \ V(x, t) = \inf_{ (\tau_i, \xi_i) \in {\cal V}[t, T]} J[x, t, \tau_i, \xi_i].
\end{equation}

In order for $J$ and $V$ to be well defined, and for the Brownian motion $W$ and the Poisson process $N$ as well as the controlled jump process $X^{x_0, t_0, \tau_i, \xi_i}$ to be unique at least in a distribution sense, we shall specify some assumptions in Section \ref{sectionassumpitons}.

The focus of the paper is to analyze the following HJB equation associated with the value function
\begin{eqnarray}\label{vis}
\mbox{(HJB)} \ \  \left\{
\begin{array}{lll}
\max \{ -u_t +L u - f - Iu, u - M u\} = 0 &\mbox{in } \mathbb{R}^n \times (0, T),\nonumber \\
u = g  &\mbox{on } \mathbb{R}^n \times \{t = T\}. \nonumber
\end{array}
\right.
\end{eqnarray}
Here
\begin{eqnarray}\label{operatorI}
I \phi(x, t) & =&\int \phi(x + j_1(x, t, z), t) - \phi(x, t) \rho(dz) \nonumber \\
& &+ \int \phi(x + j_2(x, t, z), t) - \phi(x, t) - D\phi(x, t) \cdot j_2(x, t, z) \rho(dz),
\end{eqnarray}
\begin{equation}\label{operatorM}
M u (x, t) = \inf_{\xi \in \mathbb{R}^n} ( u(x + \xi, t) + B( \xi, t) ),
\end{equation}
\begin{equation}\label{operatorL}
L u (x, t) = - tr \left[ A(x, t) \cdot D^2 u(x, t) \right] - b(x, t) \cdot D u (x, t) + r u(x, t).
\end{equation}


\paragraph{Main result.} Our main result states that the value function $V(x, t)$ is a unique $W^{(2,1), p}_{loc} (\mathbb{R}^n \times (0, T))$ viscosity solution to the  (HJB) equation  with $1, p < \infty$. In particular, for each $t \in [0, T)$, $V(\cdot, t) \in C_{loc}^{1, \gamma}(\mathbb{R}^n)$ for any $0<\gamma < 1$.

The main result is established in three steps.

\begin{itemize}
\item First, in order to connect the (HJB) equation with the value function, we  prove an appropriate form of the DPP. (Theorem \ref{DPP}).

\item Then, we show that the value function is a continuous viscosity solution to the (HJB) equation in the sense of \cite{BCI08}. (Theorem \ref{thmvis}).

\item Finally,  with additional assumptions, we show that the value function is $W_{loc}^{(2,1),p}$ for $1< p < \infty$, and in fact  a unique viscosity solution to the (HJB) equation. (Theorem \ref{thmmain}).

{\it All the results in this paper, unless otherwise specified, are built under the assumptions specified in Section \ref{sectionassumpitons}.}

\end{itemize}

\subsection{Outstanding assumptions}\label{sectionassumpitons}

\begin{assump}\label{assumpSetup}
Given $t_0 \leq T$, assume that
\begin{eqnarray*}
(\Omega_{t_0, T}, {\cal F}, \{{\cal F}_{t_0, t}[t_0, T]\}_{t_0 \leq t \leq T} ) &= &(C[t_0, T] \times M[t_0, T],\\
&& {\cal B}_{t_0, T} [t_0, T] \otimes {\cal M}_{t_0, T}[t_0, T], \\
&&\{ {\cal B}_{t_0, t} [t_0, T] \otimes {\cal M}_{t_0, t}[t_0, T] \}_{t_0 \leq t \leq T} )
\end{eqnarray*}
such that the projection map $(W, N)(x_\cdot , n) = (x_\cdot, n)$ is the Brownian motion and the Poisson point process with density $\rho(dz) \times dt$ under $\mathbb{P}$, and for $t_0 \leq s \leq t \leq T$,
\begin{eqnarray*}
C[t_0, T] &= &\{ x_\cdot : [t_0, T] \to \mathbb{R}^n,  x_{t_0} = 0\},\\
M[t_0, T] &= &\mbox{the class of locally finite measures on }[t_0, T] \times \mathbb{R}^k \setminus \{0\},\\
{\cal B}_{s, t} [t_0, T] &= &\sigma ( \{ x_r : x_\cdot \in C[t_0, T], s \leq r \leq t\}),\\
{\cal M}_{s, t} [t_0, T] &= &\sigma ( \{ n(B): B \in {\cal B}([s, t] \times \mathbb{R}^k \setminus \{0\}), n \in M[t_0, T]\}).
\end{eqnarray*}

\end{assump}

\begin{assump}\label{assumpLip}
(Lipschitz Continuity.) The functions $b$, $\sigma$, and $j$ are deterministic measurable functions such that there exists constant $C$ independent of $t \in [t_0, T]$, $z \in \mathbb{R}^k \setminus \{0\}$ such that
\begin{eqnarray*}
| b(x, t) - b(y, t) | &\leq C | x-y|,\\
|\sigma (x, t) - \sigma(y, t)| &\leq C |x-y|,\\
\int_{|z| \geq 1} |j_1(x, t, z) - j_1(x, t, z)| \rho(dz) &\leq C|x - y|,\\
\int_{|z| <1}  |j_2(x, t, z) - j_2(x, t,z)|^2 \rho(dz) &\leq C|x - y|^2.
\end{eqnarray*}
\end{assump}

\begin{assump}\label{assumpGrowth}
(Growth Condition.) There exists constant $C>0$, $\nu \in [0, 1)$, such that for any $x, y \in \mathbb{R}^n$,
\begin{eqnarray*}
|b(t, x)| &\leq &L(1+|x|^\nu),\\
|\sigma(t, x)| &\leq &L(1+|x|^{\nu/2}),\\
\int_{|z| \geq1} |j_1(x, t, z)| \nu (dz) &\leq &C (1+|x|^\nu),\\
\int_{|z| < 1} |j_2(x, t, z)|^2 \nu (dz)  &\leq &C (1+|x|^\nu).\\
\end{eqnarray*}
\end{assump}

\begin{assump}\label{assumpHolder}
(H$\ddot o$lder Continuity.) $f: [0, T] \times \mathbb{R}^n \times \mathbb{R}$ and $g: \mathbb{R}^n \to \mathbb{R}$ are measurable functions such that there exists $C>0$, $\delta \in (0, 1]$, $\gamma \in (0, \infty)$ such that
\begin{eqnarray*}
|f(t, x) - f(t, \hat x)| \leq C (1 + |x|^\gamma + | \hat x|^\gamma)  |x - \hat x|^\delta,\\
|g(x) - g( \hat x)| \leq C  (1 + |x|^\gamma + | \hat x|^\gamma) |x - \hat x|^\delta,
\end{eqnarray*}
for all $t \in [0, T]$, $x, \hat x \in \mathbb{R}^n$.
\end{assump}

\begin{assump}\label{assumpLower}
(Lower Boundedness) There exists an $L>0$ and $\mu \in (0, 1]$ such that
\begin{eqnarray*}
f(t, x) &\geq &-L,\\
h(x) &\geq &-L,\\
B(\xi, t) &\geq &L + C|\xi|^\mu,
\end{eqnarray*}
for all $t \in [0, T]$, $x \in \mathbb{R}^n$, $\xi \in \mathbb{R}^n$.
\end{assump}

\begin{assump}\label{assumpStruct}
(Monotonicity and Subadditivity) $B: \mathbb{R}^n \times [0, T] \to \mathbb{R}$ is a continuous function such that for any $0 \leq s \leq t \leq T$, $B(t, \xi) \leq B(s, \xi)$, and for $(t, \xi), (t, \hat \xi)$ being in a fixed compact subset of $\mathbb{R}^n \times [0, T)$, there exists constant $K>0$ such that
\begin{eqnarray*}
B(t, \xi + \hat \xi) +K \leq &B(t, \xi) + B(t, \hat \xi).
\end{eqnarray*}
\end{assump}

\begin{assump}\label{assumpDom}
(Dominance) The growth of $B$ exceeds the growth of the cost functions $f$ and $g$  so that 
\begin{eqnarray*}
\delta + \gamma < &\mu,\\
\nu \leq &\mu.
\end{eqnarray*}
\end{assump}

\begin{assump}\label{assumpTerm}
(No Terminal Impulse) For any $x, \xi \in \mathbb{R}^n$,
\begin{eqnarray*}
g(x) \leq \inf_\xi g(x + \xi) + B(\xi, T).
\end{eqnarray*}
\end{assump}

\begin{assump}\label{assumpInt}
Suppose that there exists a measurable map $\pi: \mathbb{R}^n \times [0, T] \to M(\mathbb{R}^n \setminus \{0\})$, in which $M(\mathbb{R}^n \setminus \{0\})$ is the set of locally finite measure on $\mathbb{R}^n \setminus \{0\}$, such that one has the following representation of the integro operator:
\begin{eqnarray*}
I \phi (x, t) = &\int [\phi(x + z, t) - \phi(x, t) - D\phi(x, t) \cdot z 1_{|z| \leq 1}] \pi(x, t, dz).
\end{eqnarray*}
And assume that for $(x, t)$ in some compact subset of $\mathbb{R}^n \times [0, T]$, there exists $C$ such that
\begin{eqnarray*}
\int_{|z|<1} |z|^2 \pi(x, t, dz) + \int_{|z|\geq 1} |z|^{\gamma + \delta} \pi(x, t, dz) \leq C.
\end{eqnarray*}
\end{assump}


\paragraph{Notations}
Throughout the paper,  unless otherwise specified, we will use the following notations.
\begin{itemize}

\item $0 < \alpha \leq 1$.

\item $$A(x, t) = (a_{ij})_{n \times n} (x, t) = \frac{1}{2} \sigma(x, t) \sigma(x, t)^T.$$

\item $\Xi (x, t)$ is the set of points $\xi$ for which $M V $ achieves the value, i.e.,
\begin{eqnarray*}
\Xi(x, t) = \{ \xi \in \mathbb{R}^n: M V (x, t) = V(x + \xi, t) + B(\xi, t)\}.
\end{eqnarray*}
\item The continuation region ${\cal C}$ and the action region ${\cal A}$ are
\begin{eqnarray*}
{\cal C} &:=& \{ (x, t) \in \mathbb{R}^n \times [0, T]: V(x, t) < M V(x, t)\},\\
{\cal A} &:=& \{ (x, t) \in \mathbb{R}^n \times [0, T]: V(x, t) = M V(x, t)\}.
\end{eqnarray*}

\item Let $\Omega$ be a bounded open set in $\mathbb{R}^{n+1}$. Denote $\partial_P \Omega$ to be the parabolic boundary of $\Omega$, which is the set of points $(x_0, t_0) \in \overline \Omega$ such that for all $R>0$, $Q(x_0, t_0; R) \not\subset \overline \Omega$. Here $
Q (x_0, t_0; R)= \{ (x, t) \in \mathbb{R}^{n+1}: \max \{ |x - x_0|, |t-t_0|^{1/2} \} < R, t < t_0 \}$.

Note that $\overline \Omega$ is the closure of the open set $\Omega$ in $\mathbb{R}^{n+1}$. In the special case of a cylinder, $\Omega = Q(x_0, t_0; R)$, the parabolic boundary $\partial_P \Omega = \left( B(0, R) \times \{ t= 0 \} \right) \cup \left( \{ |x| = R\} \times [0, T] \right)$.

\item Function spaces for $\Omega$ being a bounded open set, \begin{eqnarray*}
W^{(1, 0), p} (\Omega) &=& \{ u \in L^p(\Omega ) : u_{x_i} \in L^p(\Omega) \}, \\
W^{(2, 1), p} (\Omega) &=& \{ u \in W^{(1, 0), p} (\Omega ) : u_{x_i x_j} \in L^p(\Omega) \},\\
C^{2, 1}(\overline \Omega) &=& \{ u \in C (\overline \Omega):  u_t, u_{x_i x_j} \in C (\overline \Omega)\},\\
C^{0+\alpha, 0+\frac{\alpha}{2}} (\overline \Omega) &=& \{u \in C( \overline \Omega ): \sup_{(x, t), (y, s) \in \Omega, (x, t) \neq (y, s)}\frac{|u(x, t) - u(y, s)|}{(|x-y|^2 + |t-s|)^{\alpha/2}} < \infty \},\\
C^{2+\alpha, 1+\frac{\alpha}{2}} (\overline \Omega) &=& \{u \in C( \overline \Omega ): u_{x_i x_j}, u_t \in C^{0+\alpha, 0+\frac{\alpha}{2}}(\overline \Omega) \},\\
L^p_{loc} (\Omega) &=& \{ u|_{U} \in L^p(U) \,\,\,\forall \mbox{open } U \mbox{ such that } \overline U \subset \overline \Omega \setminus \partial_P \Omega \},\\
 W^{(1, 0), p}_{loc} (\Omega) &=& \{ u \in L^p_{loc} (\Omega): u \in W^{(1, 0), p} (U) \,\, \forall \mbox{open } U \mbox{ such that } \overline U \subset \overline \Omega \setminus \partial_P \Omega \},\\
W^{(2, 1), p}_{loc} (\Omega) &=& \{ u \in L^p_{loc} (\Omega): u \in W^{(2, 1), p} (U) \,\, \forall \mbox{open } U \mbox{ such that } \overline U \subset \overline \Omega \setminus \partial_P \Omega \}.
\end{eqnarray*}

\end{itemize}

\section{Dynamic Programming Principle and Some Preliminary Results}

\subsection{Dynamic Programming Principle}

\begin{theorem}({\bf Dynamic Programming Principle}) \label{DPP}
Under Assumptions 1-7, for $t_0 \in [0, T]$, $x_0 \in \mathbb{R}^n$, let $\tau$ be a stopping time on $(\Omega_{t_0, T}, \{{\cal F}_{t_0,t}\}_t)$, we have
\begin{eqnarray}\label{DPP1}
V (t_0, x_0) &= &\inf_{(\tau_i, \xi_i) \in {\cal V}[t_0, T]}\mathbb{E} \left[ \int_{t_0}^{\tau \wedge T} f(s, X^{t_0, x_0,\tau_i, \xi_i }_s) ds\right]\nonumber\\
 &&+ \mathbb{E} \left[\sum_i B(\xi_i, \tau_i) 1_{\tau_i \leq \tau \wedge T}+ V (\tau \wedge T, X^{t_0, x_0, \tau_i, \xi_i }_{\tau \wedge T})\right].
\end{eqnarray}
\end{theorem}

In order to establish the DPP, the first key issue is: given a stopping time $\tau$, how the martingale property and the stochastic integral change under the regular conditional probability distribution $(\mathbb{P}|{\cal F}_\tau)$.
The next key issue is the continuity of the value function, which will ensure that a countable selection is adequate without the abstract measurable selection theorem. (See \cite{Evstigneev76}).

To start, let us first introduce a new function that connects two Brownian paths which start from the origin at different times into a single Brownian path. This function also combines two Poisson measures on different intervals into a single Poisson measure.

\begin{definition}
For each $t \in [t_0, T]$, define a map $\Pi^t = (\Pi^t_1, \Pi^t_2): C[t_0, T] \times M[t_0, T] \to C[t_0, t] \times M[t_0, t] \times C[t, T] \times M[t, T]$ such that
\begin{eqnarray*}
\Pi_1^t(x_\cdot, n) &= &(x|_{[t_0, t]}, n|_{[t_0, t] \times \mathbb{R}^k  \setminus \{0\}}),\\
\Pi_2^t(x_\cdot, n) &= &(x|_{[t, T]}- x_t, n|_{(t, T]\times \mathbb{R}^k  \setminus \{0\}}).
\end{eqnarray*}
\end{definition}
Note that this is an ${\cal F}_{t_0, T}[t_0, T] / {\cal F}_{t_0, t}[t_0, t] \otimes {\cal F}_{t, T}[t, T]$-measurable bijection. Therefore, for fixed $(y_\cdot, m) \in C[t_0, T] \times M[t_0, T]$, the map from $C[t, T] \times M[t, T] \to C[t_0, T] \times M[t_0, T]$ defined by
\begin{eqnarray*}
(x_\cdot, n) &\mapsto &(\Pi^t)^{-1} (\Pi^t_1(y_\cdot, m), \Pi^t_2 (x_\cdot, n))\\
&& = (x_{\cdot \vee t} - x_t + y_{\cdot \wedge t}, m|_{[t_0, t] \times \mathbb{R}^k  \setminus \{0\}} + n|_{(t, T]\times \mathbb{R}^k  \setminus \{0\}})
\end{eqnarray*}
is ${\cal F}_{t, s}[t, T] / {\cal F}_{t_0, s} [t_0, T]$-measurable for each $s \in [t , T]$.

Next, we need two technical lemmas regarding $(\mathbb{P}|{\cal F}_\tau)$. Specifically, the first lemma states that  the local martingale property is  preserved, and the second one ensures that the stochastic integration is well defined under $(\mathbb{P}|{\cal F}_\tau)$.

According to Theorem 1.2.10 of \cite{StroockVaradhan06},
\begin{lemma}\label{stoppingMG}
 Given a filtered space, $(\Omega, {\cal F}, \{{\cal F}_t\}_{0 \leq t \leq T}, \mathbb{P})$, and an associated  martingale $\{ M_t\}_{ 0 \leq t \leq T}$. Let $\tau$ be an ${\cal F}$-stopping time. Assume $(\mathbb{P}|{\cal F}_\tau)$ exists. Then, for $\mathbb{P}$-a.e. $\omega \in \Omega$, $N_t = M_t - M_{t \wedge \tau}$ is a local martingale under $(\mathbb{P}|{\cal F}_\tau)(\omega, \cdot)$.
\end{lemma}

\begin{lemma}\label{intpreserv}
Given a filtered space  $(\Omega, {\cal F}, \{{\cal F}_t\}_{0 \leq t \leq T}, \mathbb{P})$, a stopping time $\tau$, a  previsible process $H: (0, T] \times \Omega \to \mathbb{R}^n$,  a local martingale $M: [0, T] \times \Omega \to \mathbb{R}^n$ such that
\begin{eqnarray*}
\int_\tau^T |H_s|^2 d [M]_s < \infty
\end{eqnarray*}
$\mathbb{P}$-almost surely, and $N_t = \int_\tau^t H_s dM_s$ (a version of the stochastic integral that is right-continuous on all paths). Assume that $(\mathbb{P}|{\cal F}_\tau)$ exists. Then, for $\mathbb{P}$-a.e. $\omega \in \Omega$, $N_t$ is also the stochastic integral $ \int_\tau^t H_s dM_s$ under the new probability measure $(\mathbb{P}|{\cal G})(\omega, \cdot)$.
\end{lemma}

The proof is elementary and is listed in  the Appendix for completeness.

\smallskip
Now, we establish the first step of the Dynamic Programming Principle.

\begin{proposition}\label{preDPP}
Let $\tau$ be a stopping time defined on some setup $(\Omega, \{{\cal F}_{t_0,s}\})$. For any impulse control $(\tau_i, \xi_i) \in {\cal V}[t_0, T]$,
\begin{eqnarray}\label{preDPPeq}
J[t_0, x_0,\tau_i, \xi_i] &=& \mathbb{E} \left[ \int_{t_0}^{\tau \wedge T} f(s, X^{t_0, x_0, \tau_i, \xi_i}_s) ds + \sum_i B(\xi_i, \tau_i) 1_{\tau_i < \tau \wedge T}\right]\nonumber\\
&& + \mathbb{E}\left[J[\tau \wedge T, X^{t_0, x_0, \tau^\omega_i, \xi^\omega_i}_{\tau \wedge T -}, \tau^\omega_i, \xi^\omega_i ] \right].
\end{eqnarray}
 Here $\tau^\omega_i, \xi^\omega_i$ are defined as follows. For $t \in [t_0, T]$, for each $(y_\cdot, m) \in C[t_0, T] \times M ([t_0, T] \times \mathbb{R}^k \setminus \{0\})$,
\begin{eqnarray*}
\tau_i^{y_\cdot, m} (x_\cdot, n) = &\tau_i ( (\Pi^t)^{-1} (\Pi^t_1(y_\cdot, n), \Pi^t_2 (x_\cdot, n))),\\
\xi_i^{y_\cdot, m} (x_\cdot, n) = &\xi_i ( (\Pi^t)^{-1} (\Pi^t_1(y_\cdot, n), \Pi^t_2 (x_\cdot, n))).
\end{eqnarray*}
And for each $\omega$,
 \begin{eqnarray*}
\tau_i^\omega = \tau_i^{\tau(\omega), W_\cdot (\omega), N(\omega)},\\
\xi_i^\omega = \xi_i^{\tau(\omega), W_\cdot (\omega), N(\omega)}.
\end{eqnarray*}
\end{proposition}

\begin{proof}
Consider $(\mathbb{P}|{\cal F}_{t_0, \tau})$ on $(\Omega_{t_0, t}, \{{\cal F}_{t_0, t}\} )$. Since we are working with canonical spaces, the sample space is in fact a Polish space (see \cite{Kallenberg02} Theorem A2.1 and A2.3), and the regular conditional probability exists by Theorem 6.3 of \cite{Kallenberg02}. Since Polish spaces are completely separable metric spaces and have countably generated $\sigma$-algebra, ${\cal F}_{t_0, \tau}$ is countably generated. By Lemma 1.3.3 from Stroock \& Varadhan \cite{StroockVaradhan06}, there exists some null set $N_0$ such that if $(x_\cdot, n) \notin N_0$, then
\begin{eqnarray*}
(\mathbb{P}|{\cal F}_{t_0, \tau}) ((x_\cdot, n), \{(y_\cdot, n) : \Pi_1^{\tau(x_\cdot, n)}(y_\cdot, n)= \Pi_1^{\tau(x_\cdot, n)}(x_\cdot, n) \}) = 1.
\end{eqnarray*}
Therefore, for $\omega = (x_\cdot ,n) \notin N_0$, $\tau_i = \tau_i^\omega$, and $\xi_i = \xi_i^\omega$ almost surely.

Moreover, by Lemma \ref{intpreserv}, the stochastic integrals are preserved. Therefore, for $\omega \notin N_0$, the solution to Eq. (\ref{LevySDE}) remains a solution to the same equation on the interval $[\tau(\omega), T]$ with $(\tau_i^\omega, \xi_i^\omega) \in {\cal V}[\tau(\omega), T]$. So $X^{t_0, x_0, \tau_i, \xi_i}$ on the interval $[\tau(\omega), T]$ has the same distribution as $X^{\tau(\omega), y, \tau_i^\omega, \xi_i^\omega}$ for $y = X^{t_0, x_0, u_\cdot, \tau_i, \xi_i}_{\tau(\omega)}(\omega)$ under $(\mathbb{P}|{\cal F}_{t_0, \tau})(\omega, \cdot)$ for $\omega \notin N_0$.
\end{proof}

Now, to obtain the Dynamic Programming Principle, one needs to take the infimum on both sides of Eq. (\ref{preDPPeq}). The part of ``$\leq$'' is immediate, but the opposite direction is more delicate. At the stopping time $\tau$, for each $\omega$, one needs to choose a good control so that the cost $J$ is close to the optimal $V$. To do this,  one needs to show that the functional $J$ is continuous in some sense, and therefore  a countable selection is adequate.

The following result, the H$\ddot o$lder continuity of the value function, is essentially Theorem 3.1 of Tang \& Yong \cite{TY93}. The major difference is that their work is for diffusions without jumps, therefore some modification in terms of estimation and adaptedness are needed, as outlined in the proof.

\begin{lemma}\label{Holder}
There exists constant $C>0$ such that for all $t, \hat t \in [0, T]$, $x, \hat x \in \mathbb{R}^n$,
\begin{eqnarray*}
-C(T+1) &\leq &V(t, x) \leq C(1+|x|^{\gamma + \delta}),\\
|V(t, x) - V(\hat t, \hat x)| &\leq &C[(1 + |x|^\mu + |\hat x|^\mu)|t - \hat t|^{\delta/2} + (1 + |x|^\gamma + |\hat x|^\gamma)|x - \hat x|^{\delta} ].
\end{eqnarray*}
\end{lemma}

\begin{proof}
To include the jump terms, it suffices to note the following inequalities,
\begin{eqnarray*}
\mathbb{E} \left| \int_{t_0}^t \int j_1(s, X_s, z)  N(dz, ds)\right|^\beta &\leq &\mathbb{E} \left(\int_{t_0}^t \int | j_1(s, X_s, z) | \rho(dz) ds \right)^{\beta},\\
\mathbb{E} \left| \int_{t_0}^t \int j_2(s, X_s, z) \widetilde{N}(dz, ds)\right|^\beta &\leq &\mathbb{E} \left(\int_{t_0}^t \int |j_2(s, X_s, z) |^2\rho(dz) ds \right)^{\beta/2}.
\end{eqnarray*}
Moreover, in our framework, $\bar \xi(\cdot)$ and $\hat \xi(\cdot)$ would not be in ${\cal V}[\hat t, T]$ because it is adapted to the filtration $\{{\cal F}^{W, N}_{t, s}\}_{\hat t \leq s \leq T}$ instead of $\{{\cal F}^{W, N}_{\hat t, s}\}_{\hat t \leq s \leq T}$. To fix this, consider for each $\omega \in \Omega_{t_0, T}$,
\begin{eqnarray*}
\bar \xi^\omega(\cdot) = &\bar \xi((\Pi^{\hat t})^{-1} (\Pi^{\hat t}_1 (\omega), \Pi^{\hat t}_2(\cdot))),\\
\hat \xi^\omega(\cdot) = &\hat \xi((\Pi^{\hat t})^{-1} (\Pi^{\hat t}_1 (\omega), \Pi^{\hat t}_2(\cdot))),\\
\end{eqnarray*}
 and consequently use $\mathbb{E} \left[ J[\hat t, x, \bar \xi^\omega]\right]$ instead of $J[\hat t, x, u_\cdot, \bar \xi]$.
\end{proof}

Given that the value function $V$ is continuous, we can prove Theorem \ref{DPP}.

\begin{proof}{(Dynamic Programming Principle)}
Without loss of generality, assume that $\tau \leq T$.
\begin{eqnarray*}
J[t_0, x_0, \tau_i, \xi_i] &= &\mathbb{E} \left[ \int_{t_0}^\tau f(s, X^{t_0, x_0}_s) ds + \sum_i B(\tau_i, \xi_i) 1_{\tau_i < \tau}\right]\\
&&+\mathbb{E}\left[J[\tau, X^{t_0, x_0, u^\omega_\cdot, \tau^\omega_i, \xi^\omega_i}_{\tau-}, u^{\omega}_\cdot, \tau^\omega_i, \xi^\omega_i ] \right] \\
&\geq &\mathbb{E} \left[ \int_{t_0}^\tau f(s, X^{t_0, x_0,\tau_i^\omega, \xi_i^\omega}_s) ds + \sum_i B(\tau_i, \xi_i) 1_{\tau_i < \tau}\right]\\
&& +\mathbb{E}\left[ V(\tau -, X^{t_0, x_0, \tau_i^\omega, \xi_i^\omega}_{\tau-})\right].
\end{eqnarray*}
Taking infimum on both sides, we get
\begin{eqnarray*}
V (t_0, x_0) \geq \inf_{(\tau_i, \xi_i) \in {\cal V}_{t_0}}\mathbb{E} \left[ \int_{t_0}^{\tau} f(s, X^{t_0, x_0, \tau_i^\omega, \xi_i^\omega}_s) ds +  \sum_i B(\tau_i, \xi_i) 1_{\tau_i \leq \tau} + V(\tau, X^{t_0, x_0, \tau_i^\omega, \xi_i^\omega}_{\tau})\right].
\end{eqnarray*}

Now we are to prove the reverse direction for the above inequality.

Fix $\epsilon > 0$. Divide $\mathbb{R}^n \times [t_0, T)$ into rectangles $\{R_j \times [s_j, t_j)\}$ disjoint up to boundaries, such that for any $x, \hat x \in R_j$ and $t, \hat t \in [s_j, t_j)$,
\begin{eqnarray*}
|V(x, t) - V(\hat x, \hat t)| < &\epsilon,\\
|t_j - s_j| < &\epsilon,\\
diam (R_j) < &\epsilon.
\end{eqnarray*}
For each $R_j$, pick $x_j \in R_j$. For each $(x_j, t_j)$, choose $u^j \in {\cal U}[t_j, T]$, $(\tau^j_k, \xi^j_k) \in {\cal V}[t_j, T]$, such that $V(x_j, t_j) + \epsilon > J[x_j, t_j, \tau^j_i, \xi^j_i]$. Let
\begin{eqnarray*}
A_j = \{ (X^{t_0, x_0, \tau_i, \xi_i}_{\tau^j}, \tau^j) \in R_j \times (s_j, t_j] \}
\end{eqnarray*}
Note that $\{A_j\}_j$ partitions the sample space $C[t_0, T] \times M[t_0, T]$. Define a new stopping time $\hat \tau$ by:
\begin{eqnarray*}
\hat \tau = t_j \ \  &\mbox{on } A_j.
\end{eqnarray*}
Note that $\hat \tau > \tau$ unless $\hat \tau = T$.

Define a new strategy $(\hat \tau_i, \xi_i) \in {\cal V}[t_0, T]$ by the following,

\begin{eqnarray*}
\sum_i \hat \xi_i 1_{\hat \tau_i \leq t} =
\left\{
\begin{array}{ll}
\sum_i \xi_i 1_{\tau_i \leq t} &\mbox{if } t \leq \tau,\\
\sum_i \xi_i 1_{\tau_i \leq \tau} + \sum_{i, j} 1_{A_j} (\xi^j_i 1_{\tau^j_i \leq t})(\Pi^{t_j}_2 (W, N)) &\mbox{if } t > \tau.
\end{array}
\right.
\end{eqnarray*}
In other word, once $\tau$ is reached, the impulse will be modified so that there would be no impulses on $[\tau, \hat \tau)$, and starting at $\hat \tau$, the impulse follows the rule $(\tau^j_i, \xi^j_i)$ on the set $A_j$. Now we have,

\begin{eqnarray*}
V(t_0, x_0) &\leq &J[t_0, x_0, \hat \tau_i, \hat \xi_i ] \\
&= &\mathbb{E} \left[ \int_{t_0}^{\hat \tau} f(s, X^{t_0, x_0, \hat \tau_i, \hat \xi_i}_s)) ds + \sum_i B(\hat \tau_i, \hat \xi_i) 1_{\hat \tau_i < \hat \tau} + J \left[ \hat \tau, X_{\hat \tau-}^{t_0, x_0, \hat \tau_i, \hat \xi_i}, \hat \tau_i^\omega, \hat \xi_i^\omega \right]\right] \\
&= &\mathbb{E} \left[ \int_{t_0}^{\tau} f(s, X^{t_0, x_0, \tau_i, \xi_i}_s)) ds \right] + \mathbb{E} \left[ \int_{\tau}^{\hat \tau} f(s, X^{t_0, x_0, \hat \tau_i, \hat \xi_i}_s)) ds \right] + \mathbb{E} \left[ \sum_i B(\tau_i, \xi_i) 1_{\tau_i < \tau} \right]\\
&&+ \mathbb{E} \left[  \sum_{j} J \left[ t_j, X_{t_j-}^{t_0, x_0, \hat \tau_i, \hat \xi_i}, \hat \xi_i^\omega \right] 1_{A_j} \right].
\end{eqnarray*}

The last equality follows from the fact that, $\hat \tau_i$ is either $< \tau$, or $\geq \hat \tau$, so $\hat \tau_i < \hat \tau$ implies that $\hat \tau_i = \tau_i < \tau$. Since $\hat u^\omega_\cdot = u_\cdot (\Pi^{-1} (\Pi_1^{\hat \tau}(W(\omega), N(\omega)), \Pi_2(\cdot))) = u^j_\cdot (\cdot)$ on the set $A_j$,

\begin{eqnarray*}
V(t_0, x_0) &\leq &\mathbb{E} \left[ \int_{t_0}^{\tau} f(s, X^{t_0, x_0, \tau_i, \xi_i}_s, u_s)) ds \right] + \mathbb{E} \left[ \int_{\tau}^{\hat \tau} f(s, X^{t_0, x_0, \hat \tau_i, \hat \xi_i}_s)) ds \right]\\
&& + \mathbb{E} \left[ \sum_i B(\tau_i, \xi_i) 1_{\tau_i \leq \tau} \right] + \mathbb{E} \left[  \sum_{j} J \left[ t_j, X_{t_j - }^{t_0, x_0, \hat \tau_i, \hat \xi_i}, \tau^j_i, \xi_i^j \right] 1_{A_j} \right]\\
&\leq &\mathbb{E} \left[ \int_{t_0}^{\tau} f(s, X^{t_0, x_0, \tau_i, \xi_i}_s)) ds \right] + \mathbb{E} \left[ \int_{\tau}^{\hat \tau} f(s, X^{t_0, x_0, \hat \tau_i, \hat \xi_i}_s)) ds \right]\\
&&  + \mathbb{E} \left[ \sum_i B(\tau_i, \xi_i) 1_{\tau_i \leq \tau} \right] + \mathbb{E} \left[  \sum_{j} V( t_j, X_{t_j - }^{t_0, x_0, \hat \tau_i, \hat \xi_i})) 1_{A_j} \right] + \epsilon.
\end{eqnarray*}

Now, for the second term in the last expression, we see
\begin{eqnarray}
&&\mathbb{E} \left[ \int_{\tau}^{\hat \tau} f(s, X^{t_0, x_0, \hat \tau_i, \hat \xi_i}_s)) ds \right] \leq \mathbb{E} \left[ \int_\tau^{\hat \tau} C (1+|X_s|^{\gamma + \delta})ds\right] \nonumber\\
&\leq &\mathbb{E}  \left[ \int_\tau^{\hat \tau} C (1+\mathbb{E}^\omega|X_s - X_\tau|^{\gamma + \delta} + |X_\tau|^{\gamma + \delta} )ds\right]\nonumber\\
&\leq &\mathbb{E}  \left[ \int_\tau^{\hat \tau} C (1 + |X_\tau|^{\mu} )ds\right]
\leq \mathbb{E}  \left[\mathbb{E}^\omega \left[ \int_\tau^{\hat \tau} C (1 + |X_\tau|^{\mu} )ds\right]\right] \nonumber\\
&\leq &\mathbb{E} \left[\epsilon C (1 + |X_\tau|^{\mu} ) \right]
\leq C \epsilon (1 + |x_0|^{\mu} ).
\end{eqnarray}

Therefore, it suffices to bound the following expression,
\[
\mathbb{E} \left[  \sum_{j} [V( t_j, X_{t_j - }^{t_0, x_0, \hat u_\cdot, \hat \tau_i, \hat \xi_i})- V( \tau, X_{\tau - }^{t_0, x_0, u_\cdot, \tau_i, \xi_i})] 1_{A_j} \right].
\]

First, note that on the interval $[\tau, \hat \tau)$, $X = X^{t_0, x_0, \hat u_\cdot, \hat \tau_i, \hat \xi_i}$ solves the jump SDE with no impulse:
\begin{eqnarray*}
X_{t \wedge \tau} - X_{\tau} &=& \int_\tau^{t \wedge \hat \tau}  b(s, X_s) dt +\int_\tau^{t \wedge \hat \tau}  \sigma(s, X_s)  dW\\
 && + \int_\tau^{t \wedge \hat \tau} \int j_1(s, X_s, z) \tilde{N}(dz, ds)\\
 && + \int_\tau^{t \wedge \hat \tau} \int j_2(s, X_s, z) N(dz, ds).
\end{eqnarray*}
In particular, under $(\mathbb{P} |( {\cal F}^\circ)^{W, N}_{t_0, \tau})$, $\tau$, $X^{t_0, x_0, \hat \tau_i, \hat \xi_i}_\tau$ and $\hat \tau$ are all deterministic, hence the following estimates
\begin{eqnarray*}
\mathbb{E}^{(\mathbb{P} |( {\cal F}^\circ)^{W, N}_{t_0, \tau})(\omega)}|X_{\hat \tau (\omega)}|^\beta &\leq &C (1 + |X_{\tau(\omega)}|^\beta), \quad\quad {\rm if}\ \beta>0,\\
\mathbb{E}^{(\mathbb{P} |( {\cal F}^\circ)^{W, N}_{t_0, \tau})(\omega)}|X_{\hat \tau (\omega)}- X_{\tau(\omega)}|^\beta &\leq &C (1 + |X_{\tau(\omega)}|^\beta)(\hat \tau(\omega) - \tau(\omega))^{\beta/2 \wedge 1}\\
&\leq &C (1 + |X_{\tau(\omega)}|^\beta)\epsilon^{\beta/2 \wedge 1},  \quad\quad {\rm if }\ \beta \geq \nu.\\
\end{eqnarray*}

Thus, let $E^\omega = \mathbb{E}^{(\mathbb{P} |( {\cal F}^\circ)^{W, N}_{t_0, \tau})(\omega)}$, we see
\nonumber
\begin{eqnarray*}
&&\mathbb{E}^\omega [V( t_j, X_{t_j - }^{t_0, x_0, \hat \tau_i, \hat \xi_i})- V( \tau, X_{\tau - }^{t_0, x_0, \tau_i, \xi_i})] \\
&\leq &\mathbb{E}^\omega \left[ C(1+|X_{\tau}(\omega)|^\mu +|X_{\hat \tau}|^\mu) |\hat \tau- \tau(\omega)|^{\delta/2} +C(1+|X_{\tau}(\omega)|^\gamma +|X_{\hat \tau}|^\gamma) |X_{\hat \tau} - X_{\tau}(\omega)|^{\delta} \right]\\
&\leq &C(1+|X_{\tau}(\omega)|^\mu + \mathbb{E}^\omega |X_{\hat \tau}|^\mu) \epsilon^{\delta/2} +C\mathbb{E}^\omega \left[(1+|X_{\tau}(\omega)|^\gamma +|X_{\hat \tau}|^\gamma) |X_{\hat \tau} - X_{\tau}(\omega)|^{\delta}\right]\\
&\leq &C(1+|X_{\tau}(\omega)|^\mu) \epsilon^{\delta/2} +C\left(\mathbb{E}^\omega \left[(1+|X_{\tau}(\omega)|^\gamma +|X_{\hat \tau-}|^\gamma \right]^{p'} \right)^{1/p'} \left( \mathbb{E}^\omega \left[|X_{\hat \tau} - X_{\tau}(\omega)|^{\delta}\right]^p \right)^{1/p}
\end{eqnarray*}
(where $p = \mu/\delta > 0$, and $1/p + 1/p' = 1$)
\nonumber
\begin{eqnarray*}
&\leq &C(1+|X_{\tau}(\omega)|^\mu) \epsilon^{\delta/2} +C\left(1+|X_{\tau}(\omega)|^\gamma +\left(\mathbb{E}^\omega|X_{\hat \tau-}|^{\gamma p'} \right)^{1/p'} \right)\left( \mathbb{E}^\omega |X_{\hat \tau} - X_{\tau}(\omega)|^{\mu} \right)^{\delta/\mu}\\
&\leq &C(1+|X_{\tau}(\omega)|^\mu) \epsilon^{\delta/2} + C\left(1+|X_{\tau}(\omega)|^\gamma \right) (1+ |X_{\tau(\omega)}|^\mu)^{\delta/\mu} \epsilon^{(\nu/2 \wedge 1)\frac{\delta}{\mu}}\\
&\leq &C(1+|X_{\tau}(\omega)|^\mu) \epsilon^{\delta/2} + C\left(1+|X_{\tau}(\omega)|^\mu \right) \epsilon^{(\nu/2 \wedge 1)\frac{\delta}{\mu}}.
\end{eqnarray*}
Taking expectation, we get
\begin{eqnarray}
&&\mathbb{E} \left[  \sum_{j} [V( t_j, X_{t_j - }^{t_0, x_0, \hat \tau_i, \hat \xi_i})- V( \tau, X_{\tau - }^{t_0, x_0, \tau_i, \xi_i})] 1_{A_j} \right] \nonumber \\
&\leq &C(1+\mathbb{E}|X_{\tau}|^\mu) \epsilon^{\delta/2} + C\left(1+\mathbb{E}|X_{\tau}|^\mu \right) \epsilon^{(\nu/2 \wedge 1)\frac{\delta}{\mu}}\nonumber \\
&\leq &C(1+|x_0|^\mu) \epsilon^{\delta/2} + C\left(1+|x_0|^\mu \right) \epsilon^{(\nu/2 \wedge 1)\frac{\delta}{\mu}}.
\end{eqnarray}
The last inequality follows from Corollary 3.7 in  Tang \& Yong \cite{TY93}.

With these two bounds, and taking $\epsilon\to 0$, we get the desired inequality and the DPP.

\end{proof}

\subsection{Preliminary Results}
To analyze the value function, we also need some preliminary results, in addition to the DPP.
\begin{lemma}
The set
\[
\Xi (x, t) := \{ \xi \in \mathbb{R}^n: M V(x, t) = V(x + \xi, t) + B(\xi, t)\}
\]
is nonempty, i.e. the infimum is in fact a minimum. Moreover, for $(x, t)$ in bounded $B' \subset\mathbb{R}^n \times [0, T]$, $\{ (y, t) : y = x+\xi, (x, t) \in B', \xi \in \Xi(x, t)\}$ is also bounded.
\end{lemma}

\begin{proof}
This is easy by $B(\xi, t) \geq L + C|\xi|^\mu$, $-C \leq V \leq C(1+|x|^{\gamma + \delta})$, and $\mu > \gamma + \delta$.
\end{proof}

\begin{lemma}\label{reg1}(Theorem 4.9 in \cite{Lieberman96})
Assume that $a_{ij}, b_i, f \in C^\alpha_{loc}(\mathbb{R}^n \times (0, T))$. If $-u_t + Lu = f$ in ${\cal C}$ in the viscosity sense, then it solves the PDE in the classical sense as well, and $u(x, T-t) \in C^{2+\alpha, 1+\frac{\alpha}{2}}_{loc} ({\cal C})$.
\end{lemma}


\begin{lemma}
The value function $V$ and $M V$ satisfies $V(x, t) \leq M V (x, t)$ pointwise.
\end{lemma}

\begin{lemma}
$M V$ is continuous, and there exists $C$ such that for any $x, y \in \mathbb{R}^n$, $s<t$,
\begin{eqnarray*}
|MV(x, t) - M V(y, t) | &\leq &C(1+|x|^\gamma + |y|^\gamma) |x - y|^\delta,\\
M V(x, t) - MV(x, s)  &\leq &C(1+|x|^\mu) |t - s|^{\delta/2}.
\end{eqnarray*}
\end{lemma}

\begin{proof}
First we prove continuity. For each $\xi$, $V(x, t) + B(\xi, t)$ is a uniformly continuous function on compact sets. And since $\Xi(x, t)$ is bounded for $(x, t)$ on compact sets, taking the infimum over $\xi$ on some fixed compact sets implies that $MV$ is continuous.

For the H$\ddot{o}$lder continuity in $t$, let $\xi \in \Xi(x, s)$, then
\begin{eqnarray*}
&&MV(x, t) - MV(x, s)\\
&\leq &V(x+\xi, t)+ B(\xi, t) - V(x+\xi, s) - B(\xi, s)\\
&\leq &C(1+|x|^\mu)|t-s|^{\delta/2},
\end{eqnarray*}
given that $B(\xi, s) \geq B(\xi, t)$ for $s < t$.
\end{proof}


As a consequence, the continuous region ${\cal C}$ is open.

\begin{lemma}\label{criticalLemma}
Fix $x$ in some bounded $B \subset \mathbb{R}^n$. Let $\epsilon > 0$. For any
\begin{eqnarray*}
\xi \in \Xi^\epsilon (x, t) = \{\xi: V(x + \xi, t) + B(\xi, t) < MV (x, t) + \epsilon \},
\end{eqnarray*}
we have
\begin{equation}
V(x+ \xi , t) + K - \epsilon < M V (x + \xi  , t ).
\end{equation}
In particular, Let ${\cal C}^{K/2} = \{ (x, t) \in \mathbb{R}^n \times [0, T]: V(x, t) < M V (x, t) - K/2\}$. Then if $\xi \in \Xi^{K/2}(x, t)$, then $( x+\xi , t ) \in {\cal C}^{K/2}$.
\end{lemma}

\begin{proof}
Suppose $\xi \in \Xi^\epsilon(x, t)$, i.e.
\begin{eqnarray*}
V(x + \xi, t) + B(\xi, t) <  M V(x, t) + \epsilon.
\end{eqnarray*}
Then,
\begin{eqnarray*}
M V(x + \xi, t) &= &\inf_\eta V(x + \xi + \eta, t) + B(\eta, t)\\
&= &\inf_\eta V(x + \xi + \eta, t) + B(\xi + \eta, t) - B(\xi + \eta, t) +  B(\eta, t)\\
&\geq & \inf_\eta V(x + \xi + \eta, t) +  B(\xi + \eta, t) - B(\xi, t) + K\\
&= &\inf_{\eta'} V(x + \eta', t) + B(\eta', t) - B(\xi, t) + K\\
&= &MV(x, t) - B(\xi, t) + K\\
&> &V(x + \xi, t) - \epsilon + K.
\end{eqnarray*}
Let $\epsilon = K/2$, we get that $\xi \in \Xi^{K/2}(x, t)$ implies $x + \xi \in {\cal C}^{K/2}$.
\end{proof}

\begin{lemma}\label{semiconcave}
$MV$ is uniformly semi-concave in $x$, and $MV_t$ is bounded above in the distributional sense on compact sets away from $t = T$.
\end{lemma}

\begin{proof}
Let $A$ be a compact subset of $\mathbb{R}^n \times [0, T-\delta]$. For any $\xi \in \Xi (x, t)$ for $(x, t) \in A$, $(x+\xi, t)$ lies in a bounded region $B$ independent of $(x, t)$. For any $|y| = 1$ and $\delta >0$ sufficiently small,
\begin{eqnarray*}
&&\frac{MV(x+\delta y, t) - 2 MV(x, t) + M(x- \delta y, t)}{2\delta}\\
 &\leq &\frac{(V(x+\delta y+\xi, t) + B(\xi, t)) - 2(V(x+\xi, t) + B(\xi, t)) + (V(x-\delta y + \xi, t) + B(\xi, t)))}{2 \delta}\\
 &= &\frac{V(x+\delta y+\xi, t) - 2V(x+\xi, t) +V(x-\delta y + \xi, t) }{2 \delta}\\
 &\leq &C \| D^2 V\|_{B \cap {\cal C}^{K/2}},
\end{eqnarray*}
which is bounded by Lemma \ref{reg1}. Similarly,
\begin{eqnarray*}
&&\frac{MV(x, t+\delta) - MV(x, t)}{\delta}\\
&\leq &\frac{V(x+\xi, t+\delta) +B(\xi, t+\delta) - (V(x+\xi, t) + B(\xi, t))}{\delta}\\
&= &\frac{V(x+\xi, t+\delta) - V(x + \xi, t)}{\delta} + \frac{B(\xi, t+\delta) - B(\xi, t)}{\delta}\\
&\leq &C \|V_t\|_{B \cap {\cal C}^{K/2}}.
\end{eqnarray*}
\end{proof}

\section{Value function as a Viscosity Solution}
In this section, we establish the value function $V(x,t)$ as a viscosity solution to the (HJB) equation in the  sense of
\cite{BCI08}.

\begin{theorem}({\bf Viscosity Solution of the Value Function}) \label{thmvis}
Under Assumption 1-9, the value function $V(x,t)$ is a continuous  viscosity solution to the  (HJB) equation
in the following sense: if for any $\phi \in C^2 (\mathbb{R}^n \times [0, T])$,
\begin{enumerate}
\item
$u - \phi$ achieves a local maximum at $(x_0, t_0) \in B(x_0, \theta) \times [t_0, t_0+ \theta)$ with $u(x_0, t_0)= \phi (x_0, t_0)$, then $V$ is a subsolution
\begin{eqnarray*}
\max \{ -\phi_t+ L\phi- f- I^1_\theta [\phi]  - I^2_\theta[u], u - M u \} (x_0, t_0) \leq 0.
\end{eqnarray*}
\item  $u\ge \phi$ and
$u - \phi$ achieves a local minimum at $(x_0, t_0) \in B(x_0, \theta) \times [t_0,t_0+ \theta)$ with $u(x_0, t_0)= \phi (x_0, t_0)$, then $V$ is a supersolution
\begin{eqnarray*}
\max \{ -\phi_t + L\phi - f - I^1_\theta [\phi]  - I^2_\theta[u], u - M u \}(x_0, t_0)\geq 0.
\end{eqnarray*}
\end{enumerate}
Here
\begin{eqnarray*}
I^1_\theta [\phi] (x, t)=&\int_{|z| < \theta} \phi (x+z, t) - \phi(x, t) - D\phi(x, t) \cdot z 1_{|z| < 1}\, \pi(x, t, dz),\\
I^2_\theta [u](x, t)=&\int_{|z| \geq \theta} u (x+z, t) - u(x, t) - D\phi(x, t) \cdot z 1_{|z| < 1}\, \pi(x, t, dz),
\end{eqnarray*}
with the boundary condition $u = g$ on $\mathbb{R}^n \times \{t = T\}$.
\end{theorem}

\begin{proof}
First, suppose $V - \phi$ achieves a local maximum in $B(x_0, \theta) \times [t_0, t_0 + \theta)$ with $V(x_0, t_0) = \phi(x_0, t_0)$, we prove by contradiction that $ (-\phi_t + L\phi - I^1_\theta[\phi] -I^2_\theta[V] - f )(x_0, t_0) \le 0$.

Suppose otherwise, i.e.  $ (-\phi_t + L\phi - I^1_\theta[\phi] -I^2_\theta[V] - f )(x_0, t_0) > 0$. Then without loss of generality we can assume that $-\phi_t + L\phi - I^1_\theta[\phi] -I^2_\theta[V] - f  > 0$ on $B(x_0, \theta) \times [t_0, t_0 + \theta)$. Since the definition of viscosity solution does not concern the value of $\phi$ outside of $B(x_0, \theta) \times [t_0, t_0 + \theta)$, we can assume that $\phi$ is bounded by multiples of $|V|$. Let $X^0 = X^{x_0, t_0, \infty, 0}$ and
\begin{eqnarray*}
\tau = \inf \{ t \in [t_0, T] : X_t \notin B(x_0, \theta) \times [t_0, t_0 + \theta) \} \wedge T.
\end{eqnarray*} By Ito's formula,
\begin{eqnarray*}
\mathbb{E} \left[ \phi (X^0_\tau, \tau) \right] - \phi(x_0, t_0) = &\mathbb{E} \left[ \int_{t_0}^\tau (\phi_t - L\phi + I^1_\theta [\phi] + I^2_\theta[\phi]) (X^0_{s-}, s) ds \right].
\end{eqnarray*}
Meanwhile, by Theorem\ref{DPP},
\begin{eqnarray*}
V(x_0, t_0) \leq \mathbb{E} \left[ \int_{t_0}^\tau f(X^0_s, s) ds + V(X^0_\tau, \tau) \right].
\end{eqnarray*}
Combining these two inequalities, we get
\begin{eqnarray*}
&&\mathbb{E} \left[ V (X^0_\tau, \tau) \right] - \mathbb{E} \left[ \int_{t_0}^\tau (\phi_t - L\phi + I^1_\theta [\phi] + I^2_\theta[\phi]) (X^0_{s-}, s) ds \right] \\
&\leq &\mathbb{E} \left[ \phi (X^0_\tau, \tau) \right] - \mathbb{E} \left[ \int_{t_0}^\tau (\phi_t - L\phi + I^1_\theta [\phi] + I^2_\theta[\phi]) (X^0_{s-}, s) ds \right] \\
&= &\phi(x_0, t_0) = V(x_0, t_0) \leq \mathbb{E} \left[ \int_{t_0}^\tau f(X^0_s, s) ds + V(X^0_\tau, \tau) \right].
\end{eqnarray*}
That is,
$$\mathbb{E} \left[ \int_{t_0}^\tau (-\phi_t + L\phi - I^1_\theta[\phi] -I^2_\theta[\phi] - f ) (X^0_{s-}, s) ds\right] \leq 0.
$$
Again by modifying the value of $\phi$ outside of $B(x_0, \theta) \times [t_0, t_0 + \theta)$, and since $V \leq \phi$ in $B(x_0, \theta) \times [t_0, t_0 + \theta)$, we can take a sequence of $\phi_k \geq V$ dominated by multiples of $|V|$ such that it converges to $V$ outside of $B(x_0, \theta) \times [t_0, t_0 + \theta)$ from above. By the dominated convergence theorem, $I^2_\theta[\phi]$ converges to $I^2_\theta[V]$. Thus,
\begin{eqnarray*}
\mathbb{E} \left[ \int_{t_0}^\tau (-\phi_t + L\phi - I^1_\theta[\phi] -I^2_\theta[V] - f ) (X^0_{s-}, s) ds\right] \leq 0,
\end{eqnarray*}
which is a contradiction. 
Therefore, we must have $(-\phi_t + L\phi - I^1_\theta[\phi] -I^2_\theta[V] - f  )(x_0, t_0) \leq 0$, and since $V \leq MV$, we have $\max \{ -\phi_t + L\phi - I^1_\theta[\phi] -I^2_\theta[V] - f , V - MV \} (x_0, t_0) \leq 0$.

Next,  suppose $V - \phi$ achieves local minimum in $B(x_0, \theta) \times [t_0, t_0 + \theta)$ with $V(x_0, t_0) = \phi(x_0, t_0)$. Then if
 $(V - MV )(x_0, t_0) = 0$, then we already have the desired inequality. Now
  suppose $V - MV \leq -\epsilon < 0$ and $-\phi_t + L\phi - I^1_\theta[\phi] -I^2_\theta[V] - f < 0$ on $B(x_0, \theta) \times [t_0, t_0 + \theta)$. Assuming as before that $\phi$ is bounded by multiples of $|V|$ outside of $B(x_0, \theta) \times [t_0, t_0 + \theta)$. By Ito's formula
\begin{eqnarray*}
\mathbb{E} \left[ \phi (X^0_\tau, \tau) \right] - \phi(x_0, t_0) = \mathbb{E} \left[ \int_{t_0}^\tau (\phi_t - L\phi + I^1_\theta [\phi] + I^2_\theta[V]) (X^0_{s-}, s) ds \right].
\end{eqnarray*}

Consider the no impulse strategy $\tau^*_i = \infty$ and let $X^0 = X^{t_0, x_0,\infty, 0}$. Define the stopping time $\tau$ as before, i.e, \begin{eqnarray*}
\tau = \inf \{ t \in [t_0, T] : X_t \notin B(x_0, \theta) \times [t_0, t_0 + \theta) \} \wedge T.
\end{eqnarray*}
Then for any strategy $(\tau_i, \xi_i) \in {\cal V}$,
\begin{eqnarray*}
&&J[t_0, x_0, \tau_i, \xi_i]\\
&=&\mathbb{E} \left[ \int_{t_0}^{\tau_1 \wedge \tau} f(s, X^{t_0, x_0, \tau_i, \xi_i}_s) ds + B(\tau_1, \xi_1) 1_{\{\tau_1 \leq \tau \wedge \tau_1 \}} + J[\tau_1 \wedge \tau, X^{t_0, x_0, \tau_i, \xi_i}_{\tau_1 \wedge \tau}, \tau_i, \xi_i] \right]\\
&\geq &\mathbb{E} \left[\int_{t_0}^{\tau_1 \wedge \tau} f(s, X^{t_0, x_0, \tau_i, \xi_i}_s) ds  + 1_{\{\tau_1 \leq \tau \}} (B(\tau_1, \xi_1)+V(X^{t_0, x_0, \tau_i, \xi_i}_{\tau_1}, \tau_1))\right]\\
&& +\mathbb{E} \left[1_{\{\tau_1 >\tau \}} V(X^{t_0, x_0, \tau_i, \xi_i}_{\tau}, \tau)\right]\\
&\geq &\mathbb{E} \left[\int_{t_0}^{\tau_1 \wedge \tau} f(s, X^0_s) ds  + 1_{\{\tau_1 \leq \tau \}} MV(X^0_{\tau_1},  \tau_1)\right]\\
&& +\mathbb{E} \left[1_{\{\tau_1 >\tau \}} V(X^0_{\tau}, \tau)\right]\\
&\geq &\mathbb{E} \left[\int_{t_0}^{\tau_1 \wedge \tau} f(s, X^0_s) ds  +V(X^0_{\tau_1 \wedge \tau},  \tau_1 \wedge \tau)\right] + \epsilon \cdot \mathbb{P} (\tau_1 \leq \tau)\\
&\geq &V(t_0, x_0)+ \epsilon \cdot \mathbb{P} (\tau_1 \leq \tau).
\end{eqnarray*}
Therefore, without loss of generality, we only need to consider $(\tau_i, \xi_i)\in {\cal V}$ such that $\tau_1 > \tau$. Now, the Dynamic Programming Principle becomes,
\begin{eqnarray*}
u(x_0, t_0) = \mathbb{E} \left[ \int_{t_0}^\tau f(X^0_s, s) ds + V(X^0_\tau, \tau) \right].
\end{eqnarray*}

Now combining these facts above,
\begin{eqnarray*}
&&\mathbb{E} \left[ V (X^0_\tau, \tau) \right] - \mathbb{E} \left[ \int_{t_0}^\tau (\phi_t - L\phi + I^1_\theta [\phi] + I^2_\theta[\phi]) (X^0_{s-}, s) ds \right] \\
&\geq &\mathbb{E} \left[ \phi (X^0_\tau, \tau) \right] - \mathbb{E} \left[ \int_{t_0}^\tau (\phi_t - L\phi + I^1_\theta [\phi] + I^2_\theta[\phi]) (X^0_{s-}, s) ds \right] \\
&= &\phi(x_0, t_0) = V(x_0, t_0) =\mathbb{E} \left[ \int_{t_0}^\tau f(X^0_s, s) ds + V(X^0_\tau, \tau) \right]\\
&&\mathbb{E} \left[ \int_{t_0}^\tau (-\phi_t + L\phi - I^1_\theta[\phi] -I^2_\theta[\phi] - f ) (X^0_{s-}, s) ds\right] \geq 0.
\end{eqnarray*}
Again by modifying the value of $\phi$ outside of $B(x_0, \theta) \times [t_0, t_0 + \theta)$, and since $V \geq \phi$ in $B(x_0, \theta) \times [t_0, t_0 + \theta)$, we can take a sequence of $\phi_k \leq u$ dominated by multiples of $|V|$ such that it converges to $V$ outside of $B(x_0, \theta) \times [t_0, t_0 + \theta)$ from above. By the dominated convergence theorem, $I^2_\theta[\phi]$ converges to $I^2_\theta[V]$. Hence
\begin{eqnarray*}
\mathbb{E} \left[ \int_{t_0}^\tau (-\phi_t + L\phi - I^1_\theta[\phi] -I^2_\theta[V] - f ) (X^0_{s-}, s) ds\right] \geq 0,
\end{eqnarray*}
 which contradicts the assumption that $(-\phi_t + L\phi - I^1_\theta[\phi] -I^2_\theta[V] - f )(x_0, t_0)< 0$. Therefore, $\max \{-\phi_t + L\phi - I^1_\theta[\phi] -I^2_\theta[V] - f, V - MV \}(x_0, t_0) \geq 0$.
\end{proof}



\section{Regularity of the Value Function}

 To study the regularity of the value function, we will consider the time-inverted value function $u(x, t) = V(x, T-t)$. Accordingly, we will assume that $a_{ij}$, $b_i$, $f$, $B$ and $j$ are all time-inverted. This is to be  consistent with the standard PDE literature for easy references to some of its classical results, where the value is specified at the initial time instead of the terminal time.

The regularity study is built in two phases.

First in Section \ref{secnojumps}, we focus on the case without jumps . We will construct a unique $W^{(2,1),p}_{loc}$ regular viscosity solution to a corresponding equation without the integro-differential operator on  a fixed bounded domain $Q_T$ with $Q_T = B(0, R) \times (\delta, T]$ for $R > 0$ and $\delta>0$,
\begin{eqnarray}
\left\{
\begin{array}{ll}
\max \{u_t + Lu - f, u - \Psi \} = 0 &\mbox{in } Q_T,\\
u(t, x) = \phi &\mbox{on } \partial_P Q_T.
\end{array}
\right.
\end{eqnarray}
in which $\phi(x, t) = V(x, T-t)$ and $\Psi (x, t)=( M u)(x, t)$.
The local uniqueness of the viscosity solution then implies that this solution must be the time-inverted value function, hence the $W^{(2, 1), p}_{loc}$ smoothness for the value function.


Then in Section \ref{regsec1}, we extend the analysis to the case with a first-order jump and establish the regularity property  of the value function.

\subsection{$W^{(2,1),p}_{loc}$ Regularity for cases without jumps}\label{secnojumps}

The key idea is to study a corresponding homogenous HJB, based on the following classical result in PDEs.

\begin{lemma} (Theorem 4.9, 5.9, 5.10, and 6.33 of \cite{Lieberman96}) \label{LinearThm}
Let $\alpha \in (0, 1]$. Assume that $a_{ij}, b_i, f \in C^{0 + \alpha, 0 + \frac{\alpha}{2}}(\overline{Q_T})$, $a_{ij}$ is uniformly elliptic, $\phi \in C^{0 + \alpha, 0 + \frac{\alpha}{2}} (\partial_P Q_T)$. Then the linear PDE
\begin{eqnarray}\label{Linear}
\left\{
\begin{array}{ll}
u_t + Lu = f &\mbox{in } Q_T;\\
u = \phi &\mbox{on }  \partial_P Q_T.
\end{array}
\right.
\end{eqnarray}
has a unique solution to (\ref{Linear}) that lies in $C^{0+\alpha, 0+\frac{\alpha}{2}}(\overline{Q_T}) \cap C_{loc}^{2+\alpha, 1+\frac{\alpha}{2}}( Q_T)$.
\end{lemma}

Indeed, given Lemma \ref{LinearThm},  let $u_0$ be the unique classical solution to (\ref{Linear}), with the boundary condition $\phi(x, t) = V(T-t, x)$. Then, our earlier analysis (Lemma \ref{Holder}) of H$\ddot o$lder continuity for the value function   implies that   $V(x, T-t)-u_0(x, t)$ solves the following ``homogenous'' HJB,
\begin{eqnarray*}
\left\{
\begin{array}{ll}
\max \{ u_t + Lu , u- \overline \Psi \} = 0  &\mbox{in } Q_T,\\
u = 0 &\mbox{on } \partial_P Q_T.
\end{array}
\right.
\end{eqnarray*}
for $\overline \Psi = \Psi - u_0$. Since $V \leq MV$, we have $\overline \Psi (x, t) = (\Psi - u_0)(x, t) = (MV - V)(x, T-t) \geq 0$ on $\partial_P Q_T$.

Therefore, our first step is to study the above ``homogenous'' HJB.

{\em Step I: Viscosity solution of the ``homogenous'' HJB}

\begin{theorem}\label{thoeremcorrect}
Assume
\begin{enumerate}
\item
$a_{ij}, b_i, \overline \Psi \in C^{0 + \alpha, 0 + \frac{\alpha}{2}}(\overline{Q_T})$,
\item
$(a_{ij})$ uniformly elliptic,
\item
$\overline \Psi$ is semiconcave,
\item
$\overline \Psi_t$ is bounded below, in the distributional sense.
\end{enumerate}
Then there exists a  viscosity solution $u \in W^{(2,1), p}(Q_T)$ to the homogenous HJB
\begin{eqnarray}\label{HJBhomogenous}
\left\{
\begin{array}{ll}
\max \{ u_t + Lu , u - \overline \Psi \} = 0 &\mbox{in } Q_T,\\
u = 0 &\mbox{on }  \partial_P Q_T.
\end{array}
\right.
\end{eqnarray}
In fact, $u \in W^{(2,1), p}(Q_T)$ for any $p > 1$.
\end{theorem}

To prove this theorem, we  first  consider a corresponding penalized version.
For every $\epsilon > 0$, let $\beta_\epsilon: \mathbb{R} \to \mathbb{R}$ be a smooth function such that $\beta_\epsilon (x) \geq  -1$, $\beta_\epsilon (0) = 0$, $\beta' > 0$, $\beta'' \geq 0$, $\beta_\epsilon '(x) \leq C/\epsilon$ for $x \geq 0$, $\beta_\epsilon'(0) = 1/\epsilon$ and as $\epsilon \to 0$, $\beta_\epsilon(x) \to \infty$ for $x >0$, $\beta_\epsilon (x) \to 0$ for $x<0$. One such example is, $\beta(x) = x/ \epsilon$ for $x \geq 0$ and it smooth extension to $x < 0$.
We see that there is a classical solution $u$ to the penalized problem, assuming some regularity on the coefficients $a^{ij}, b^i, \overline \Psi$.
\begin{lemma}\label{theorempenalized}
Fix $\epsilon>0$. Suppose that $a^{ij}, b^i, \overline \Psi \in C^{2+\alpha, 1+\alpha/2} (\overline Q_T)$, and $(a^{ij})$ is uniformly elliptic. Then exists a unique $u \in C^{4+\alpha, 2+\alpha/2} (\overline Q_T)$ such that
\begin{eqnarray}\label{penalized}
\left\{
\begin{array}{ll}
u_t + Lu + \beta_\epsilon (u - \overline \Psi)= 0&\mbox{on  } Q_T,\\
u = 0 &\mbox{on  } \partial_P Q_T.
\end{array}
\right.
\end{eqnarray}
\end{lemma}
Note that Friedman \cite{Friedman82} proved under different assumptions for a $W^{2,p}$ solution for the elliptic case using the $L^p$ estimates. He then used the H$\ddot o$lder estimates to bootstrap for the $C^2$ regularity. Our proof  uses  the Schauder estimates. (For details, see Appendix B).

Next, consider the case with $C^{\alpha, \alpha/2}(\overline Q_T)$ coefficients. We will smooth out the coefficients first to the above  result, and then let $\epsilon \to 0$. More precisely,
let $(a^\epsilon)^{ij}, (b^\epsilon)^i, \overline \Psi^\epsilon \in C^\infty(\overline Q_T)$ be such that they converge to the respective function in $C^{\alpha, \alpha/2}(\overline Q_T)$ and  $\overline \Psi^\epsilon \geq 0$ on $\partial_P Q_T$. This is possible because $\overline \Psi \geq 0$ on $\partial_P Q_T$. Define $L^\epsilon$ to be the corresponding linear operator and  $u^\epsilon$ to be the unique solution to
\begin{eqnarray}
\left\{
\begin{array}{ll}
u^\epsilon_t + L^\epsilon u^\epsilon + \beta_\epsilon (u^\epsilon - \overline \Psi^\epsilon)= 0&\mbox{on  } Q_T,\\
u^\epsilon = 0 &\mbox{on  } \partial_P Q_T.
\end{array}
\right.
\end{eqnarray}

Now we establish some bound for $\beta_\epsilon (u^\epsilon - \overline \Psi^\epsilon)$, in order to apply an $L^p$ estimate.
\begin{lemma}\label{boundedbeta}
Assuming $\overline \Psi$ is semiconcave in $x$, i.e.
\begin{eqnarray}
\frac{\partial^2 \overline \Psi}{\partial \xi^2} &\leq C,
\end{eqnarray}
for any direction $|\xi| = 1$, and
\begin{eqnarray}
\frac{\partial \overline \Psi}{\partial t} &\geq - C,
\end{eqnarray}
where both derivatives are interpreted in the distributional sense. We  have
\begin{eqnarray}
|\beta_\epsilon (u^\epsilon - \overline \Psi^\epsilon)| \leq C,
\end{eqnarray}
with $C$ independent of $\epsilon$.
\end{lemma}

\begin{proof}
Clearly $\beta_\epsilon \geq -1$, so we only need to give an upper bound. The assumption above translates to the same derivative condition on mollified $\overline \Psi^\epsilon$, which can be interpreted classically now. Thus we have
\begin{eqnarray}
\overline \Psi^\epsilon_t + L^\epsilon \overline \Psi^\epsilon \geq -C.
\end{eqnarray}
Suppose $u^\epsilon - \overline \Psi^\epsilon$ achieves maximum at $(x_0, t_0) \in Q_T$, then
\begin{eqnarray}
(u^\epsilon - \overline \Psi^\epsilon)_t + L^\epsilon (u^\epsilon - \overline \Psi^\epsilon) (x_0, t_0)\geq 0.
\end{eqnarray}
Hence
\begin{eqnarray}
- \beta_\epsilon (u^\epsilon - \overline \Psi^\epsilon) (x_0, t_0) &=& (u^\epsilon_t + L^\epsilon u^\epsilon)  (x_0, t_0)\\
&\geq &(\overline \Psi^\epsilon_t + L^\epsilon \overline \Psi^\epsilon)(x_0, t_0) \geq -C,
\end{eqnarray}
in which $C$ is an upper bound independent of $\epsilon$. On the other hand, if it achieves maximum on $\partial_P Q_T$, we get $u^\epsilon - \overline \Psi^\epsilon \leq 0$ since $\overline \Psi^\epsilon \geq 0$ on $\partial_P Q_T$. Either way we have an upper bound independent of $\epsilon$.
\end{proof}

Now with this estimate  of the boundedness of $\beta_\epsilon (u^\epsilon - \overline \Psi^\epsilon)$,
 we are ready to prove Theorem \ref{thoeremcorrect}.

\begin{proof}
Lemma \ref{boundedbeta} allows us to apply $L^p$ estimate:
\begin{eqnarray}
\|u^\epsilon \|_{W^{(2,1),p}(Q_T)} \leq C \| \beta_\epsilon (u^\epsilon - \overline \Psi^\epsilon) \|_{L^p (Q_T)} \leq C,
\end{eqnarray}
for $p > 1$. Thus there exists a sequence $\epsilon_n \to 0$ and $u \in W^{(2,1),p}(Q_T)$ such that
\begin{eqnarray}
u^{\epsilon_n} \rightharpoonup u
\end{eqnarray}
weakly in $W^{(2,1),p}(Q_T)$. For $p$ large enough, there exists $\alpha' > 0$ such that $u^\epsilon \to u$ in $C^{\alpha', \alpha'/2}(\overline Q_T)$, so $u^\epsilon \to u$ uniformly in $\overline Q_T$.

 On one hand, since $\beta_\epsilon (u^\epsilon - \overline \Psi^\epsilon) \leq C$, yet $\beta_\epsilon (x) \to \infty$ as $x>0$, hence $u \leq \overline \Psi$. Suppose $u - \phi$ achieves a strict local maximum at $(x_0, t_0)$, then $u^\epsilon - \phi$ achieves a strict local maximum at $(x^\epsilon_0, t^\epsilon_0)$ and $ (x^\epsilon_0, t^\epsilon_0) \to (x_0, t_0)$ as $\epsilon \to 0$, then
\begin{eqnarray*}
\lim_{\epsilon \to 0}(\phi_t + L^\epsilon \phi) (x^\epsilon_0, t^\epsilon_0) \leq \liminf_{\epsilon \to 0} \beta_\epsilon (u^\epsilon (x^\epsilon_0, t^\epsilon_0) - \overline \Psi (x^\epsilon_0, t^\epsilon_0)) \leq 0.
\end{eqnarray*}
So $(\phi_t + L \phi) (x_0, t_0) \leq 0$.

On the other hand, if $u - \phi$ achieves a strict local minimum at $(x_0, t_0)$, then $u^\epsilon - \phi$ achieves a strict local maximum at $(x^\epsilon_0, t^\epsilon_0)$ and $ (x^\epsilon_0, t^\epsilon_0) \to (x_0, t_0)$ as $\epsilon \to 0$. If $u(x_0, t_0) < \overline \Psi (x_0, t_0)$, then for small $\epsilon$, $u(x^\epsilon_0, t^\epsilon_0) <  \overline \Psi (x^\epsilon_0, t^\epsilon_0)$,
\begin{eqnarray*}
\lim_{\epsilon \to 0}(\phi_t + L^\epsilon \phi) (x^\epsilon_0, t^\epsilon_0) \geq \limsup_{\epsilon \to 0} \beta_\epsilon (u^\epsilon (x^\epsilon_0, t^\epsilon_0) - \overline \Psi (x^\epsilon_0, t^\epsilon_0)) \geq 0.
\end{eqnarray*}
\end{proof}

{\em Step II: Uniqueness of the HJB equation without  jump terms}

\begin{proposition}\label{unique}
 Assuming that $a_{ij}, b_i, f, \Psi, f$ are continuous in $\overline Q_T$, and $\phi$ continuous on $\partial_P Q_T$, the viscosity solution to the following HJB equation is unique.
\begin{eqnarray}\label{boundedvis}
\left\{
\begin{array}{ll}
\max \{u_t + Lu - f, u - \Psi \} = 0 &\mbox{in } Q_T,\\
u(t, x) = \phi &\mbox{on } \partial_P Q_T.
\end{array}
\right.
\end{eqnarray}
\end{proposition}

\emph{Remark.} Note that this is a local uniqueness of the viscosity solution.
We  later apply $\phi(x, t) = V(x, T-t)$ and $\Psi (x, t)=( M u)(x, t)$ to our original control problem.

\begin{proof}

Let $W, U$ be a viscosity subsolution and supersolution to (\ref{boundedvis}) respectively. Then $W$ is clearly a viscosity subsolution to $v_t + Lv - f = 0$, with $W \leq \overline \Psi$. On the other hand, at any fixed point $(x_0, t_0)$, either $U(x_0, t_0) = \overline \Psi (x_0, t_0)$ or $U$ satisfies the viscosity supersolution property at $(x_0, t_0)$.

Define
\begin{eqnarray}
W^\epsilon (x, t) = W(x, t) + \frac{\epsilon}{t-\delta}
\end{eqnarray}
for $\epsilon >0$. Note that $W^\epsilon$ is still a viscosity subsolution of $v_t+Lv - f = 0$. For fixed $\epsilon, \alpha, \beta$, define
\begin{eqnarray}
\Phi(t, x,y) = W^\epsilon (x, t) - U(x,t) - \alpha |x-y|^2 - \beta (t - \delta).
\end{eqnarray}
Denote $B = B(0, R)$. Suppose $\max_{(x, t) \in \overline Q_T} W^\epsilon (x, t) - U(x, t) \geq c >0$. There exist $\alpha_0, \beta_0, \epsilon_0$, such that for $\alpha \geq \alpha_0$, $\beta \leq \beta_0$, and $\epsilon \leq \epsilon_0$, we have
\begin{eqnarray}
\max_{(t, x, y) \in [\delta, T) \times \overline{B} \times \overline{B}} \Phi(t, x, y) \geq c/2 > 0.
\end{eqnarray}
Let $(\bar t, \bar x, \bar y) \in (\delta, T) \times B \times B$ be the point where $\Phi$ achieves the maximum. Since $\Phi(\delta, 0, 0) \leq \Phi(\bar t, \bar x, \bar y)$, we get
\begin{eqnarray}
\alpha | \bar x - \bar y |^2 \leq h( |\bar x - \bar y| ),
\end{eqnarray}
in which $h$ is the modulus of continuity of $U$. Since the domain is bounded,  $\alpha | \bar x - \bar y|^2 \leq K$ for some fixed constant $K$ independent of $\alpha, \epsilon, \beta$. We have $| \bar x - \bar y | \leq \sqrt{K/\alpha}$, which implies
\begin{eqnarray}
\alpha | \bar x - \bar y|^2 \leq \omega (\sqrt {\frac{K}{\alpha}} ).
\end{eqnarray}
Denote $\omega$ as the modulus of continuity of $\overline \Psi$. We have two cases:
\begin{enumerate}
\item $U(\bar y, \bar t) = \overline \Psi (\bar y, \bar t)$. We have
\begin{eqnarray*}
W^\epsilon (\bar x, \bar t) &\leq &\overline \Psi (\bar x, \bar t) + \frac{\epsilon}{\bar t - \delta}\\
&\leq &\omega (| \bar x - \bar y| ) + \overline \Psi (\bar y, \bar t) + \frac{\epsilon}{\bar t - \delta}\\
&= &\omega (| \bar x - \bar y| ) + U (\bar y, \bar t) + \frac{\epsilon}{\bar t - \delta}.
\end{eqnarray*}
Thus
\begin{eqnarray*}
W^\epsilon (\bar x, \bar t) - U(\bar y, \bar t) \leq \omega( | \bar x - \bar y| ) + \frac{\epsilon}{\bar t - \delta}.
\end{eqnarray*}
\item $U(\bar y, \bar t) < \overline \Psi (\bar y, \bar t)$.  By the same analysis as Theorem V.8.1 in \cite{FS06},
\begin{eqnarray}
\beta \leq \omega( \alpha |\bar x - \bar y|^2 + | \bar x - \bar y|).
\end{eqnarray}
\end{enumerate}
Fix $\epsilon \leq \epsilon_0$, $\beta \geq \beta_0$. For each $\alpha \leq \alpha_0$, one of the two cases is true. If case 2 occurs infinitely many times as $\alpha \to \infty$, we have a contradiction, thus case 1 must occur infinitely many times as $\alpha \to \infty$.  We have the inequality
\begin{eqnarray*}
&&W^\epsilon (x, t) - U(x, t) - \beta( t- \delta) \\
&=& \Phi (x, x, t) \leq \Phi (\bar x, \bar y, \bar t) \\
&\leq &W^\epsilon (\bar x, \bar t) - U(\bar y, \bar t)\\
&\leq &\omega( \sqrt{\frac{1}{\alpha} h(\sqrt{{\frac{K}{\alpha}}})}) + \frac{\epsilon}{\bar t - \delta}.
\end{eqnarray*}
Let $\alpha \to \infty$, then $W^\epsilon (x, t) - U(x, t) - \beta( t- \delta) \leq \frac{\epsilon}{\bar t - \delta}$. Let $\beta \to 0$, then $\epsilon \to 0$, we get $W(x, t) \leq U(x, t)$.
\end{proof}

Now, combining Theorem \ref{thoeremcorrect} and Proposition \ref{unique}, together  with Lemma \ref{LinearThm} for the $C(\overline{Q_T}) \cap C_{loc}^{2+\alpha, 1+\frac{\alpha}{2}}( Q_T)$ solution to (\ref{Linear}), we have
{\em Step III: Regularity  of the (HJB) equation without jump terms} 

\begin{proposition}\label{W21p}
Assume
\begin{enumerate}
\item
$a_{ij}, b_i, \Psi \in C^{0 + \alpha, 0 + \frac{\alpha}{2}}(\overline{Q_T}),$
\item
$(a_{ij})$ uniformly elliptic,
\item
$\Psi$ is semiconcave,
\item
$\Psi_t$ is bounded below, in the distributional sense,
\item
$\phi, f \in C^{0+\alpha, 0+\alpha/2}(\overline Q_T)$.

\end{enumerate}
Then there exists a unique viscosity solution $u \in W^{(2,1), p}_{loc}(Q_T) \cap C(\overline Q_T)$ to the PDE
\begin{eqnarray}
\left\{
\begin{array}{ll}
\max \{ u_t + Lu-f, u - \Psi \} = 0 &\mbox{in } Q_T,\\
u = \phi&\mbox{on }  \partial_P Q_T,
\end{array}
\right.
\end{eqnarray}
for any $p>1$.
\end{proposition}

Finally, since $M u(x, t)$ is semi-concave from  Lemma \ref{semiconcave}, replacing  $\Psi (x, t)$ by  $M u(x, t)$ gives us the regularity property of the value function.

\begin{theorem}\label{thm2p}
Under Assumption 1-8, and no jumps ($\pi(x,t,dz) = 0$), the value function $V(x, t)$ is a $W^{(2,1), p}_{loc} (\mathbb{R}^n \times (0, T))$ viscosity solution to the  (HJB) equation  with $1 < p < \infty$. In particular, for each $t \in [0, T)$, $V(\cdot, t) \in C_{loc}^{1, \gamma}(\mathbb{R}^n)$ for any $0<\gamma < 1$.
\end{theorem}


 In fact, if one adds additional assumption of   $a_{ij}$ and $b_i$  in $W^{(2,1), \infty}_{loc} (\mathbb{R}^n \times [0, T])$, then with more detailed and somewhat tedious analysis, one can establish $W^{(2, 1), \infty}_{loc}$ regularity for the value function. For more details, see
 Chen \cite{Chen12}.

\subsection{Regularity for Value Function of First-Order Jump Diffusion}\label{regsec1}

To proceed, we will  add two new assumptions in this subsection.

\begin{assump}\label{add1}
The operator $I[\phi]$ is of order-$\delta$, i.e, for $(x, t)$ in any compact subset of $\mathbb{R}^n \times [0, T]$, there exists $C$ such that
\begin{eqnarray*}
\int_{|z| <1} |z|^\delta \, \pi(x, t, dz) < C < \infty.
\end{eqnarray*}
\end{assump}

\emph{Remark.} Since $\delta \in (0, 1]$, this implies that $\int_{|z|<1} |z| \, \pi(x, t, dz) < C$.

\begin{assump}\label{add3}
The measure $\pi(x, t, dz)$ is continuous with respect to the weighted total variation, i.e., for $(x_n, t_n) \to (x_0, t_0)$,
\begin{eqnarray*}
\int \left( |z|^\gamma + |z|^\delta \right) |\pi(x_n, t_n, dz) - \pi(x_0, t_0, dz)| \to 0.
\end{eqnarray*}

\end{assump}

\begin{proposition}
With additional assumptions \ref{add1}  and \ref{add3}, there exists a unique $u \in W^{(2,1),p}_{loc}(Q_T) \cap C(\overline Q_T)$ viscosity solution of the following equations,
\begin{eqnarray}\label{vis2}
\left\{
\begin{array}{lll}
\max \{ -u_t +L u - f - Iu, u - M V\} = 0 &\mbox{in } Q_T,\\
u = V(x, T-t) &\mbox{on } \partial_P Q_T,
\end{array}
\right.
\end{eqnarray}
in the following sense. For any $\phi \in C^2 (\mathbb{R}^n \times [0, T])$,
\begin{enumerate}
\item
If $u - \phi$ achieves a local maximum at $(x_0, t_0) \in Q_T$, then
\begin{eqnarray*}
\max \{ -\phi_t+ L\phi- f- I^0 [V], u - M V \} (x_0, t_0) \leq 0;
\end{eqnarray*}
\item
If $u - \phi$ achieves a local minimum at $(x_0, t_0) \in Q_T$, then
\begin{eqnarray*}
\max \{ -\phi_t + L\phi - f - I^0 [V] , u - M V \}(x_0, t_0)\geq 0.
\end{eqnarray*}
\end{enumerate}
Here
\begin{eqnarray*}
I^0 [V] (x, t)=&\int [V (x+z, t) - V(x, t)- D\phi(x, t) \cdot z 1_{|z| < 1} ] \pi(x, t, dz),
\end{eqnarray*}
with the boundary condition $u = g$ on $\mathbb{R}^n \times \{t = T\}$.
\end{proposition}

\begin{proof}
Let $\overline b_i = b_i - \int z_i M(x, t, dz)$ and $\overline f = \int u(x+z, t) - u(x, t) M(x, t, dz)$.we will show that $\overline b_i$ and $\overline f$ are continuous.

Step 1,  $\overline f$ is continuous:

Let $x_n \to x_0$, then
\begin{eqnarray*}
&&\left| \int [V(x_n + z, t) - V(x_n, t)] \, \pi(x_n, t, dz) - \int [V(x_0 + z, t) - V(x_0, t)] \, \pi(x_0, t, dz) \right|\\
&\leq &\left| \int \left( V(x_n + z, t) - V(x_n, t)\right) \left( \pi(x_n, t, dz) - \pi(x_0, t, dz) \right) \right| \\
&&+ \left| \int \left( V(x_n + z, t) - V(x_n, t)\right) - \left( V(x_0 + z, t) - V(x_0, t) \right) \pi(x_0, t, dz) \right|.
\end{eqnarray*}
For the first term,
\begin{eqnarray*}
&&\left| \int \left( V(x_n + z, t) - V(x_n, t)\right) \left( \pi(x_n, t, dz) - \pi(x_0, t, dz) \right) \right|\\
&\leq &C \int \left( 1 + |x_n|^\gamma + |z|^\gamma \right) |z|^\delta |\pi(x_n, t, dz) - \pi(x_0, t, dz)|\\
&\leq &C \int |z|^\gamma + |z|^\delta |\pi(x_n, t, dz) - \pi(x_0, t, dz)| \to 0.
\end{eqnarray*}
For the second term, the integrand $\to 0$ as $n \to \infty$. So by the dominated convergence theorem,
\begin{eqnarray*}
&&\left| \int \left( V(x_n + z, t) - V(x_n, t)\right) - \left( V(x_0 + z, t) - V(x_0, t) \right) \pi(x_0, t, dz) \right|\\
&\leq &\int C(1+|x_n|^\gamma + |z|^\gamma ) |z|^\delta + C(1+|x_0|^\gamma + |z|^\gamma ) |z|^\delta \pi(x_0,t, dz) \\
&\leq & C \int ( |z|^\gamma + |z|^\delta ) \pi(x_0, t, dz) < \infty.
\end{eqnarray*}
Therefore $\overline f$ is continuous in $x$. Now let $t_n \to t_0$,
\begin{eqnarray*}
&&\left| \int V(x + z, t_n) - V(x, t_n) \pi(x, t_n, dz) - \int V(x + z, t_0) - V(x, t_0) \pi(x, t_0, dz) \right|\\
&\leq &\left| \int \left( V(x + z, t_n) - V(x, t_n)\right) \left( \pi(x, t_n, dz) - \pi(x, t_0, dz) \right) \right| \\
&&+ \left| \int \left( V(x + z, t_n) - V(x, t_n)\right) - \left( V(x + z, t_0) - V(x, t_0) \right) \pi(x_0, t, dz) \right|.
\end{eqnarray*}
For the first term,
\begin{eqnarray*}
&&\left| \int \left( V(x + z, t_n) - V(x, t_n)\right) \left( \pi(x, t_n, dz) - \pi(x, t_0, dz) \right) \right|\\
&\leq &\int C(1+|x|^\gamma + |z|^\gamma) |z|^\delta \left| \pi(x, t_n, dz) - \pi(x, t_0, dz) \right|\\
&\leq &C \int |z|^\gamma + |z|^\delta \left| \pi(x, t_n, dz) - \pi(x, t_0, dz) \right| \to 0,
\end{eqnarray*}
as $t_n \to t_0$. For the second term, the dominated convergence theorem implies
\begin{eqnarray*}
&&\left| \int \left[ \left( V(x + z, t_n) - V(x, t_n)\right) - \left( V(x + z, t_0) - V(x, t_0) \right) \right] \pi(x_0, t, dz) \right|\\
&\leq &\int C(1+|x|^\gamma + |z|^\gamma) |z|^\delta \pi(x, t_0, dz) < \infty.
\end{eqnarray*}
Therefore $\overline f$ is continuous in $t$.

Step 2, $\overline b_i$ is continuous:

This follows easily from Assumption \ref{add3}. Let $(x_n, t_n) \to (x_0, t_0)$,
\begin{eqnarray*}
\left| \int_{|z|<1} z [\pi(x_n, t_n, dz) - \pi(x_0, t_0, dz)] \right| \leq &\int_{|z|<1} |z| |\pi(x_n, t_n, dz) - \pi(x_0, t_0, dz)|\\
\leq &\int |z|^\delta  |\pi(x_n, t_n, dz) - \pi(x_0, t_0, dz)|,
\end{eqnarray*}
which goes to $0$ as $n \to \infty$.

Step 3, replace $b_i$ by $\overline b_i = b_i - \int z_i \pi(x, t, dz)$ and $f$ by  $\overline f = f + \int (u(x+z, t) - u(x, t)) \pi(x, t, dz)$, and by Proposition \ref{W21p}, $u \in W^{(2,1),p}_{lot}(Q_T) \cap C(\overline Q_T)$. The only caveat is the H$\ddot o$lder continuity assumption on $f$ is not satisfied here. But careful checking of our analysis for the case without jumps in Section \ref{secnojumps} reveals that we could  relax our assumption on $f$ to be bounded instead of  H$\ddot o$lder continuous, and take $f^\epsilon$ to be H$\ddot o$lder continuous and converge to $f$ in $L^\infty$.

\end{proof}

Notice, however,  the ``apparent'' difference between the two types of  viscosity solutions: the one in the above proposition, and the one  in Theorem \ref{thmvis}. Therefore, we need to show that the viscosity solution  in Theorem \ref{thmvis} is also a viscosity solution of Eqn. (\ref{vis2}). Then, with the  standard local uniqueness of HJB of Eqn. (\ref{vis2}),  the regularity of value function is obtained.

\begin{theorem}
Under Assumptions 1-11, $V$ is also a solution of Eqn. (\ref{vis2}), with additional assumptions \ref{add1} and \ref{add3}.
\end{theorem}

\begin{proof}
Suppose $V - \phi$ has a local minimum in $B(x_0, \theta_0) \times [t_0, t+\theta_0)$. Then we know that
\begin{eqnarray*}
\max\{ -\phi_t + L\phi - f - I^1_\theta[\phi] - I^2_\theta[V] , V - MV\} \leq 0
\end{eqnarray*}
for any $0 < \theta < \theta_0$. And with the additional assumptions, $I^1_\theta[\phi] + I^2_\theta[V] \to I^0 [V] $ as $\theta \to 0$. Therefore we have
\begin{eqnarray*}
\max\{ -\phi_t + L\phi - f - I^0[V] , V - MV\} \leq 0
\end{eqnarray*}
The other inequalities can be derived similarly.
\end{proof}

In summary,
\begin{theorem}\label{thmmain}
({\bf Regularity of the Value Function and Uniqueness}) 
Under Assumptions 1-11, the value function $V(x, t)$ is a unique $W^{(2,1), p}_{loc} (\mathbb{R}^n \times (0, T))$ viscosity solution to the (HJB) equation  with $2\leq p < \infty$. In particular, for each $t \in [0, T)$, $V(\cdot, t) \in C_{loc}^{1, \gamma}(\mathbb{R}^n)$ for any $0<\gamma < 1$.
\end{theorem}

\section{Appendix}

\subsection{Appendix A: Proof of Lemma \ref{intpreserv}}
\begin{proof}
First assume that $\int_\tau^T H^2 d[M]$ is bounded and $M$ is a $L^2$-martingale. The conclusion is clearly true when $H$ is an elementary previsible process of the form $\sum_i^m Z_i 1_{(S_i, T_i]}$. Now let $H^{(n)}$ be a sequence of such elementary previsible process that converges to $H$ uniformly on $\Omega\times (0, T]$. Let $N^{(n)}_t =  \int_\tau^t H^{(n)}_s dM_s$, and $N_t = \int_\tau^t H_s dM_s$.

First we show that quadratic variation is preserved under regular conditional probability distribution. By Lemma \ref{stoppingMG}, $Q_t = M_t - M_{t \wedge \tau}$ is a martingale. Consider the quadratic variation $[Q]$ under $\mathbb{P}$. By definition, $Q^2_t - [Q]_t$ is a martingale. Thus by Lemma \ref{stoppingMG}, for almost every $\omega$, $Q^2_t - [Q]_t - (Q^2_{t \wedge \tau} - [Q]_{t \wedge \tau})$ is martingale under $(\mathbb{P}|{\cal F}_\tau)(\omega)$.
\begin{eqnarray*}
&&Q^2_t - [Q]_t - (Q^2_{t \wedge \tau} - [Q]_{t \wedge \tau})\\
&= &(M_t - M_{t \wedge \tau})^2 - ([M]_t - [M]_{t \wedge \tau}) \\
&&- \left( (M_{t \wedge \tau} - M_{t \wedge \tau})^2 - ([M]_{t \wedge \tau} - [M]_{t \wedge \tau})  \right)\\
&= &Q_t^2 - ([M]_t - [M]_{t \wedge \tau})
\end{eqnarray*}
This shows that under $(\mathbb{P}|{\cal F}_\tau)(\omega)$, $[Q]^{(\mathbb{P}|{\cal F}_\tau)(\omega)}_t = [M]_t - [M]_{t \wedge \tau}$. This allows us to simply write $[ \cdot ]$ instead of $[\cdot]^{(\mathbb{P}|{\cal F}_\tau)(\omega)}$. Hence,
\begin{eqnarray*}
\int_\tau^T H^2 d[M] = \int_\tau^T H^2 d[N]
\end{eqnarray*}
is bounded under $(\mathbb{P}|{\cal F}_\tau)(\omega)$.

By the definition of $H^{(n)}$,
\begin{eqnarray*}
&& \mathbb{E}^{\mathbb{P}} \left[  \liminf_n  \mathbb{E}^{(\mathbb{P}|{\cal F}_\tau)(\omega)} \left[ \int_\tau^T |H^{(n)}_s - H_s|^2 d[Q]_s \right] \right]\\
&\leq & \liminf_n   \mathbb{E}^{\mathbb{P}} \left[ \mathbb{E}^{(\mathbb{P}|{\cal F}_\tau)(\omega)} \left[ \int_\tau^T |H^{(n)}_s - H_s|^2 d[Q]_s \right] \right]\\
&= &\liminf_n \mathbb{E}^{\mathbb{P}} \left[ \int_\tau^T |H^{(n)}_s - H_s|^2 d[M]_s  \right]^2 = 0.
\end{eqnarray*}

Thus, $\liminf_n  \mathbb{E}^{(\mathbb{P}|{\cal F}_\tau)(\omega)} \left[ \int_\tau^T |H^{(n)}_s - H_s|^2 d[Q]_s \right] = 0$ for almost every $\omega$. On the other hand,
\begin{eqnarray*}
&&\mathbb{E}^{\mathbb{P}} \left[ \liminf_n \mathbb{E}^{({\mathbb{P}}|{\cal G})(\omega)} \left[N^n_T - N_T\right]^2 \right]\\
&\leq &\liminf_n \mathbb{E}^{\mathbb{P}} \left[ \mathbb{E}^{({\mathbb{P}}|{\cal G})(\omega)} \left[N^n_T - N_T\right]^2 \right]\\
&\leq &\liminf_n \mathbb{E}^{\mathbb{P}}  \left[N^n_T - N_T\right]^2 = 0.
\end{eqnarray*}

So we have
\begin{eqnarray*}
\liminf_n \mathbb{E}^{({\mathbb{P}}|{\cal G})(\omega)} \left[ (N^n - N)_T\right]^2 = 0,
\end{eqnarray*}
for $\mathbb{P}$-a.e. $\omega$. This proves the claim. The general case follows from the localization technique.
\end{proof}

\subsection{Appendix B: Proof of Lemma \ref{theorempenalized}}

\begin{proof}
Define the operator $A: C^{2+\alpha, 1+\alpha/2}(\overline Q_T) \to C^{2+\alpha, 1+\alpha/2}(\overline Q_T)$ by the following: $A[v] = u$ is the solution to the PDE
\begin{eqnarray}
\left\{
\begin{array}{ll}
u_t + Lu + \beta_\epsilon (v - \overline \Psi)= 0&\mbox{on  } Q_T,\\
u = 0 &\mbox{on  } \partial_P Q_T.
\end{array}
\right.
\end{eqnarray}
By the Schauder's estimates (Theorem 4.28 in \cite{Lieberman96}), we have,
\begin{eqnarray*}
\|u\|_{C^{2+\alpha, 1+\alpha/2} (\overline Q_T)} &\leq C \|\beta_\epsilon (v - \overline \Psi)\|_{C^{\alpha, \alpha/2} (\overline Q_T)}\\
&\leq C_\epsilon \| v - \overline \Psi \|_{C^{\alpha, \alpha/2} (\overline Q_T)}.
\end{eqnarray*}
Thus the map $A$ is clearly continuous and compact.

The next step is to show that the set $\{ u: u = \lambda A[u], 0 \leq \lambda \leq 1\}$ is bounded. Then we can apply Schaefer's Fixed Point Theorem (Theorem 9.4 in (\cite{Evans98}). Suppose $u = \lambda A[u]$ for some $0 \leq \lambda \leq 1$. Then,
\begin{eqnarray}
\left\{
\begin{array}{ll}
u_t + Lu + \lambda \beta_\epsilon (u - \overline \Psi)= 0&\mbox{on  } Q_T,\\
u = 0 &\mbox{on  } \partial_P Q_T.
\end{array}
\right.
\end{eqnarray}
Since
\begin{eqnarray*}
\|u\|_{C^{2+\alpha, 1+\alpha/2} (\overline Q_T)} &\leq& C_\epsilon \| v - \overline \Psi \|_{C^{\alpha, \alpha/2} (\overline Q_T)})\\
&\leq& C_\epsilon (\|u\|_{C^{\alpha, \alpha/2} (\overline Q_T)} + \|\overline \Psi\|_{C^{\alpha, \alpha/2} (\overline Q_T)})\\
&\leq& C_\epsilon (\|u\|_{C(\overline Q_T)}^{1/2}\|u\|^{1/2}_{C^{2+\alpha, 1+\alpha/2} (\overline Q_T)} + \|\overline \Psi\|_{C^{\alpha, \alpha/2} (\overline Q_T)})\\
&\leq &\frac{1}{2} \|u\|_{C^{2+\alpha, 1+\alpha/2} (\overline Q_T)} + C_\epsilon (\|u\|_{C(\overline Q_T)} + \|\overline \Psi\|_{C^{\alpha, \alpha/2} (\overline Q_T)}).
\end{eqnarray*}
Thus
\begin{eqnarray}
\|u\|_{C^{2+\alpha, 1+\alpha/2} (\overline Q_T)} &\leq C_\epsilon (\|u\|_{C(\overline Q_T)} + \|\overline \Psi\|_{C^{\alpha, \alpha/2} (\overline Q_T)}).
\end{eqnarray}
So we only need to bound $u$ independent of $\lambda$ now.

If $\lambda = 0$, then $u = 0$. So we can assume that $\lambda > 0$. Suppose $u$ has a maximum at $(x_0, t_0) \in Q_T$. Then, $- \lambda \beta_\epsilon (u(x_0, t_0) - \overline \Psi (x_0, t_0)) = (u_t + Lu) (x_0,t_0) \geq 0$, $\beta_\epsilon (u(x_0, t_0) - \overline \Psi (x_0, t_0)) \leq 0$, $u(x_0, t_0) \leq \overline \Psi (x_0, t_0)$, and $u = 0$ on $\partial_P Q_T$. So we get $u \leq \|\overline \Psi \|_{L^\infty(Q_T)}$.

For a lower bound, consider the open set $\Omega= \{ u < \overline \Psi\}$ in $\overline Q_T$. Since in $\Omega$, $u_t + Lu \geq 0$, $u \geq \inf_{\partial_P \Omega} u$. Yet, $\partial_P \Omega \subset \partial_P Q_T \cup \{ u \geq \overline \Psi \}$, and in both cases $u$ is bounded below. Thus we conclude that $u$ is bounded independently of $\lambda$, and $\| u\|_{L^\infty (Q_T)} \leq  \|\overline \Psi \|_{L^\infty(Q_T)}$.

Now Schaefer's Fixed Point Theorem (Theorem 9.4  in \cite{Evans98}) gives us the existence of $u \in C^{2+\alpha, 1+\alpha/2}(\overline Q_T)$ that solves (\ref{penalized}). Now we have $-\beta_\epsilon (u - \overline \Psi) \in C^{2+\alpha, 1+\alpha/2} (\overline Q_T)$. By the Schauder's estimates again, we have $u \in C^{4+\alpha, 2+\alpha/2}(\overline Q_T)$.
\end{proof}

\bibliographystyle{siam}
\bibliography{bibdata}

\end{document}